\documentclass[12pt,oneside]{amsart}
\usepackage[margin=1in]{geometry}
\usepackage{enumitem,verbatim}
\usepackage{microtype}
\usepackage{physics}
\allowdisplaybreaks

\usepackage{hyperref}
\usepackage{color}
\hypersetup{colorlinks,linkcolor=blue,citecolor=cyan}
\usepackage{amsthm}
\usepackage[nameinlink]{cleveref}
\usepackage{amsmath,amssymb}
\usepackage{tikz, tikz-cd}

\newtheorem{cor}{Corollary}
\newtheorem{lem}{Lemma}
\newtheorem{prop}{Proposition}
\newtheorem{thm}{Theorem}
\newtheorem{conj}{Conjecture}
\theoremstyle{definition}
\newtheorem{defn}{Definition}
\newtheorem{eg}{Example}

\newtheorem{qn}{Question}
\newtheorem{rmk}{Remark}

\newcommand{\qbin}[2]{\begin{bmatrix}{#1}\\ {#2}\end{bmatrix}}
\newenvironment{psmallmatrix}
  {\left(\begin{smallmatrix}}
  {\end{smallmatrix}\right)}
\newcommand{\n}{\mathord{\includegraphics[height=1.6ex]{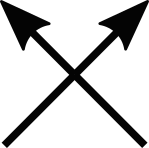}}}
\newcommand{\E}{\mathord{\includegraphics[height=1.6ex]{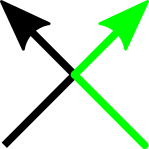}}}
\newcommand{\W}{\mathord{\includegraphics[height=1.6ex]{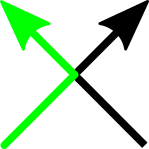}}}
\newcommand{\NW}{\mathord{\includegraphics[height=1.6ex]{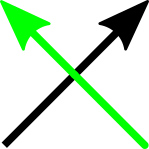}}}
\newcommand{\NE}{\mathord{\includegraphics[height=1.6ex]{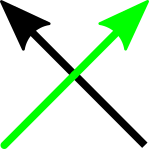}}}
\newcommand{\X}{\mathord{\includegraphics[height=1.6ex]{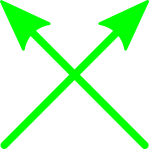}}}

\title{Inverted state sums, inverted Habiro series, and indefinite theta functions}
\author{Sunghyuk Park}
\address{Division of Physics, Mathematics and Astronomy,
California Institute of Technology,
1200 E. California Blvd.,
Pasadena,
CA 91125}
\email{spark3@caltech.edu}
\thanks{}

\begin{document}

\maketitle
\begin{abstract}
S. Gukov and C. Manolescu conjectured that the Melvin-Morton-Rozansky expansion of the colored Jones polynomials can be re-summed into a two-variable series $F_K(x,q)$, which is the knot complement version of the $3$-manifold invariant $\hat{Z}$ whose existence was predicted earlier by S. Gukov, P. Putrov and C. Vafa. In this paper we use an inverted version of the $R$-matrix state sum to prove this conjecture for a big class of links that includes all homogeneous braid links as well as all fibered knots up to 10 crossings. We also study an inverted version of Habiro's cyclotomic series that leads to a new perspective on $F_K$ and discovery of some regularized surgery formulas relating $F_K$ with $\hat{Z}$. These regularized surgery formulas are then used to deduce expressions of $\hat{Z}$ for some plumbed $3$-manifolds in terms of indefinite theta functions. 
\end{abstract}
\tableofcontents

\section{Introduction}
$\hat{Z}$ is a $q$-series-valued invariant of $3$-manifolds introduced by \cite{GPV,GPPV} from the physics of supersymmetric quantum field theories. It is a fascinating object as it is supposed to play an important role in categorification of quantum 3-manifold invariants \cite{GPPV}, and it is connected to various fields of mathematics and physics such as complex Chern-Simons theory and resurgence \cite{GMP}, quantum modular forms and logarithmic vertex algebras \cite{CCFGH}, holomorphic curve counting \cite{EGGKPS}, non-semisimple invariants and geometric representation theory \cite{GPNPPS}, and so on. 
From the purely mathematical point of view, it has been a mysterious object as well, since a fully general mathematical definition of $\hat{Z}$ is not available at the time of writing. 

An important step toward finding a mathematical home for $\hat{Z}$ was made by \cite{GM}. In their paper, Gukov and Manolescu outlined an approach to study $\hat{Z}$ by cutting and gluing. They denoted $\hat{Z}$ for the complement of a knot $K$ in $S^3$ by $F_K$. That is, 
\[
F_K := \hat{Z}(S^3\setminus K).
\]
It is a series in two variables, $x$ and $q$, and they conjectured that it can be obtained by re-summing the famous Melvin-Morton-Rozansky expansion \cite{MelvinMorton, BarNatanGaroufalidis, R} of the colored Jones polynomials, which is a power series in $\hbar = \log q$ or $q-1$, into a series in $q$. They also conjectured a set of surgery formulas that relate these $F_K$ with $\hat{Z}$ for the $3$-manifolds obtained by some Dehn surgeries on $K$. The motivation for this work was to understand $F_K$ and the surgery formulas better. 

This paper can be thought of as a sequel to a previous work by the author \cite{Park}. 
In \cite{Park}, we studied $F_K$ by using what we called the \emph{large-color $R$-matrix}. It is the large-color limit of the $R$-matrix for the $n$-colored Jones polynomials. In the limit where $n$ goes to $\infty$ while we keep $x:= q^n$ fixed, the finite-dimensional irreducible representations of $U_q(\mathfrak{sl}_2)$ become infinite dimensional Verma modules, and the large-color $R$-matrix determines the braiding of these Verma modules. 
The infinite-dimensionality of the modules poses a challenge that did not exist when $n$ was finite: The state sum may not converge. By convergence, we mean convergence as a power series in $x$ (or $x^{-1}$) and $q$. Of course there are ways to get around the convergence issue by working in a different ring such as Habiro ring and its variants \cite{Hab, Hab2, Willetts}. This leads to a beautiful and rich story in itself, but unfortunately this is not what we want for $F_K$, since $F_K$ is an ordinary power series in $x$ and $q$; we would like to keep $q$ and $x$ generic.  
The main result of \cite{Park} was that for positive braid knots, the state sum determined by the large-color $R$-matrices converges absolutely as a power series in $x^{-1}$ and $q$. In this paper, this result is extended to a much bigger class of knots and links that includes homogeneous braid links. This is done with the help of a trick that we call ``inversion". 

Inversion of a series can be best described through the following motivating example:
\[
\sum_{n\geq 0}z^n = \frac{1}{1-z} = -\frac{z^{-1}}{1-z^{-1}} = -\sum_{n<0}z^n.
\]
On the very left-hand side, we have a geometric series in $z$, which converges as long as $|z|<1$. The first equality says that the rational function $\frac{1}{1-z} = -\frac{z^{-1}}{1-z^{-1}}$, which is defined both inside and outside the unit disk, agrees with the first series inside the unit disk. On the very right-hand side, we have a power series in $z^{-1}$ that agrees with the same rational function outside the unit disk. 
In the end, we have an ``equality" between the series in the very left-hand side and the one on the very right-hand side, even though their domains of convergence are different. Note also that, up to sign, we only had to ``invert" the range of summation of $n$ from the set of non-negative integers to the set of negative integers, to go from the first series to the second one. This simple observation turns out to be an extremely useful one in our study of $F_K$. In fact, we apply this inversion technique in two different contexts, firstly to the state sum and secondly to the Habiro series. They both lead to $F_K$, so they provide two new ways to understand $F_K$.

\addtocontents{toc}{\protect\setcounter{tocdepth}{1}}
\subsection*{Organization of the paper}
This paper is divided into three parts: Section \ref{sec:invertedStateSums} on inverted state sums, Section \ref{sec:invertedHabiro} on inverted Habiro series, and Section \ref{sec:indefiniteTheta} on their connection to indefinite theta functions. 

In Section \ref{sec:invertedStateSums}, we study inverted state sums. This section contains the main theorem (Theorem \ref{thm:mainTheorem}) of this paper which says that for homogeneous braid links, the Melvin-Morton-Rozansky expansion can be re-summed into a two-variable series and therefore the conjecture of \cite{GM} is true for those links. The homogeneity condition can be weakened as in Theorem \ref{thm:beyondHomogeneousBraids}, and using this theorem we can obtain state sum expressions of $F_K$ for all fibered knots up to 10 crossings. See Appendix \ref{sec:fiberedKnotsInversionData} for the data. 

Section \ref{sec:invertedHabiro} and \ref{sec:indefiniteTheta} are more experimental in nature. In Section \ref{sec:invertedHabiro}, we invert Habiro series and find that inverted Habiro series often provides a nice closed-form expression for $F_K$. What is particularly nice about inverted Habiro series is that they suggest some regularized version of surgery formulas. \cite{GM} contains nice surgery formulas, but when applied naively they would often give a non-converging answer. The regularized surgery formulas we find in Section \ref{sec:invertedHabiro} enable us to study how $\hat{Z}$ behaves under various surgeries even when the surgery formulas of \cite{GM} are not applicable due to non-convergence. 

In Section \ref{sec:indefiniteTheta}, we show how the regularized surgery formulas are related to indefinite theta functions. An approach to $\hat{Z}$ for indefinite plumbed $3$-manifolds using indefinite theta functions was initiated by \cite{CFS}, and what we do in Section \ref{sec:indefiniteTheta} is pretty much in the same spirit. While we only provide a list of examples, we believe they will serve as a good starting point for future research in finding a formula for $\hat{Z}$ for general plumbed $3$-manifolds.

\subsection*{Notations}
Here we collect the notations and conventions we use. 
\[
[n] := \frac{q^{\frac{n}{2}}-q^{-\frac{n}{2}}}{q^{\frac{1}{2}}-q^{-\frac{1}{2}}},\quad 
[n]_q := \frac{1-q^n}{1-q}
\]
\[
[n]! := \prod_{k=1}^{n}[k],\quad 
[n]_q! := \prod_{k=1}^{n}[k]_q
\]
\[
\qbin{n}{k} := \frac{[n]!}{[k!][n-k]!},\quad 
\qbin{n}{k}_q := \frac{[n]_q!}{[k]_q![n-k]_q!} = \frac{(q;q)_n}{(q;q)_k(q;q)_{n-k}}
\]
\[
(x)_n = (x;q)_n := \prod_{k=0}^{n-1}(1-q^k x)
\]
Note, $\qbin{n}{k} = 0 = \qbin{n}{k}_q$ unless $n\geq k\geq 0$ or $0 > n\geq k$ or $k\geq 0 > n$. 

Unless otherwise specified, we work with $U_q(\mathfrak{sl}_2)$.
The highest weight Verma module with the highest weight $\lambda$ is denoted by $V^h_{\infty,\lambda}$, and the lowest weight Verma module with the lowest weight $\lambda$ is denoted by $V^l_{\infty,\lambda}$. Throughout this article, especially in the first section, we take the point of view of the highest weight Verma module $V^h_{\infty,\log_q x-1}$ which leads to a series in $x^{-1}$. If one wants to get a series in $x$, then they simply need to apply the Weyl symmetry $x\leftrightarrow x^{-1}$ to the whole story; that is, start with $V^l_{\infty,-\log_q x +1}$ instead. 

\subsection*{Acknowledgements}
I would like to express my gratitude to Sergei Gukov for encouragement and helpful conversations. I also thank Piotr Kucharski and Piotr Su\l{}kowski for their feedback on an earlier version of the draft. 

This work was partially supported by Kwanjeong Educational Foundation.

\addtocontents{toc}{\protect\setcounter{tocdepth}{2}}
\section{Inverted state sums}\label{sec:invertedStateSums}
\subsection{Extending the $R$-matrix}
In \cite{Park}, we considered the large-color limit of the $R$-matrices for $n$-colored Jones polynomials. The large-color $R$-matrix and its inverse are given by
\begin{align}
\check{R}(x)_{i,j}^{i',j'} &= \delta_{i+j,i'+j'} q^{\frac{j+j'+1}{2}}x^{-\frac{j+j'+1}{2}} q^{jj'} \qbin{i}{j'}_q \prod_{1\leq l\leq i-j'}(1-q^{j+l}x^{-1}),\\
\check{R}^{-1}(x)_{i,j}^{i',j'} &= \delta_{i+j,i'+j'} q^{-\frac{i+i'+1}{2}}x^{\frac{i+i'+1}{2}}q^{-ii'}\qbin{j}{i'}_{q^{-1}} \prod_{1\leq l\leq j-i'}(1-q^{-i-l}x). \nonumber
\end{align}
They correspond to positive and negative crossings, respectively. See Fig. \ref{fig:RandRinv}. Throughout this section, as in the figure, we will use the indices $i, j, i', j'$ to refer to the bottom left, the bottom right, the top left, and the top right strand of a crossing, respectively. 
\begin{figure}[ht]
    \centering
    \includegraphics[scale=0.6]{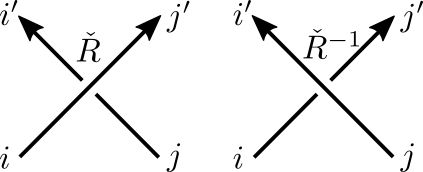}
    \caption{$\check{R}$ and $\check{R}^{-1}$ correspond to positive and negative crossings respectively. }
    \label{fig:RandRinv}
\end{figure}

In the standard state sum, the range of summation for the indices $i,j,i',j'$ is the set of non-negative integers. A crucial observation, however, is that the domain of $i,j,i',j'$ can be naturally extended to the set of all integers, since the $R$-matrix entry makes sense for any integers $i,j,i',j'$. Explicitly, 
\begin{align*}
\check{R}(x)_{i,j}^{i',j'} &= \begin{cases}\delta_{i+j,i'+j'} q^{\frac{j+j'+1}{2}}x^{-\frac{j+j'+1}{2}} q^{jj'} \qbin{i}{i-j'}_q \prod_{1\leq l\leq i-j'}(1-q^{j+l}x^{-1}) &\text{if } \begin{array}{c}i\geq j'\geq 0\\ \text{ or }\\ 0>i\geq j' \end{array}\\ \delta_{i+j,i'+j'} q^{\frac{j+j'+1}{2}}x^{-\frac{j+j'+1}{2}} q^{jj'} \qbin{i}{j'}_q \frac{1}{\prod_{0\leq l\leq j'-i-1}(1-q^{j-l}x^{-1})} &\text{if }j'\geq 0>i \\ 0 &\text{otherwise}\end{cases}
\end{align*}
and
\begin{align*}
\check{R}^{-1}(x)_{i,j}^{i',j'} &= \begin{cases}\delta_{i+j,i'+j'} q^{-\frac{i+i'+1}{2}}x^{\frac{i+i'+1}{2}} q^{-ii'} \qbin{j}{j-i'}_{q^{-1}} \prod_{1\leq l\leq j-i'}(1-q^{-i-l}x) &\text{if } \begin{array}{c}j\geq i'\geq 0\\ \text{ or }\\ 0>j\geq i'\end{array}\\ \delta_{i+j,i'+j'} q^{-\frac{i+i'+1}{2}}x^{\frac{i+i'+1}{2}} q^{-ii'} \qbin{j}{i'}_{q^{-1}} \frac{1}{\prod_{0\leq l\leq i'-j-1}(1-q^{-i+l}x)} &\text{if }i'\geq 0>j \\ 0 &\text{otherwise.}\end{cases}
\end{align*}
As power series in $x^{-1}$, we have
\begin{align}\label{eq:bigO}
\check{R}(x)_{i,j}^{i',j'} &= O(x^{-\frac{j+j'+1}{2}}),\\
\check{R}^{-1}(x)_{i,j}^{i',j'} &= O(x^{\frac{i+i'+1}{2}+(j-i')}) = O(x^{-\frac{i'-i-2j-1}{2}})=O(x^{\frac{j+j'+1}{2}}).\nonumber
\end{align}

When the two strands are different, we can similarly extend the domain. Let's say the $i\rightarrow j'$ strand is assigned with the variable $x$ and $j\rightarrow i'$ strand is assigned with the variable $y$. Then we have
\begin{align*}
\check{R}(x,y)_{i,j}^{i',j'} &= \begin{cases}\delta_{i+j,i'+j'} q^{\frac{j+j'+\frac{1}{2}}{2}}x^{-\frac{i'+j+1}{4}}y^{-\frac{3j'-i+1}{4}} q^{jj'} \qbin{i}{i-j'}_q \prod_{1\leq l\leq i-j'}(1-q^{j+l}y^{-1}) &\text{if } \begin{array}{c}i\geq j'\geq 0\\ \text{ or }\\ 0>i\geq j' \end{array}\\ \delta_{i+j,i'+j'} q^{\frac{j+j'+\frac{1}{2}}{2}}x^{-\frac{i'+j+1}{4}}y^{-\frac{3j'-i+1}{4}} q^{jj'} \qbin{i}{j'}_q  \frac{1}{\prod_{0\leq l\leq j'-i-1}(1-q^{j-l}y^{-1})} &\text{if }j'\geq 0>i \\ 0 &\text{otherwise}\end{cases}
\end{align*}
and
\begin{align*}
\check{R}^{-1}(x,y)_{i,j}^{i',j'} &= \begin{cases}\delta_{i+j,i'+j'} q^{-\frac{i+i'+\frac{1}{2}}{2}}x^{\frac{3i'-j+1}{4}}y^{\frac{j'+i+1}{4}} q^{-ii'} \qbin{j}{j-i'}_{q^{-1}} \prod_{1\leq l\leq j-i'}(1-q^{-i-l}x) &\text{if } \begin{array}{c}j\geq i'\geq 0\\ \text{ or }\\ 0>j\geq i'\end{array}\\ \delta_{i+j,i'+j'} q^{-\frac{i+i'+\frac{1}{2}}{2}}x^{\frac{3i'-j+1}{4}}y^{\frac{j'+i+1}{4}} q^{-ii'} \qbin{j}{i'}_{q^{-1}} \frac{1}{\prod_{0\leq l\leq i'-j-1}(1-q^{-i+l}x)} &\text{if }i'\geq 0>j \\ 0 &\text{otherwise.}\end{cases}
\end{align*}
Note, in our convention, $\check{R}(x) = q^{\frac{1}{4}}\check{R}(x,x)$.

\subsection{Meaning of the extension}
Before extension, `the space of spin states' assigned to each strand of a braid in our state sum model is the highest weight Verma module $V_{\infty,\log_q x-1}^h$ with the highest weight $\lambda = \log_q x-1$. The set of basis vectors $\{v^n\}_{n\geq 0}$ are labelled by non-negative integers, where each $v^n$ is an element of the weight subspace $V_\infty(\lambda-2n)$. In a specific choice of basis used in \cite{Park}, the action of $U_q(\mathfrak{sl}_2)$ is given by
\begin{align*}
    e v^n &= [n]v^{n-1},\\
    f v^{n} &= [\lambda-n]v^{n+1},\\
    q^{\frac{h}{2}}v^{n} &= q^{\frac{\lambda - 2n}{2}}v^{n}.
\end{align*}
We can describe it diagrammatically as
\[\cdots\overset{e}{\underset{f}{\rightleftharpoons}}V_\infty(\lambda-2)\overset{e}{\underset{f}{\rightleftharpoons}} V_\infty(\lambda).\]
Now, extending $n$ to negative values simply means we consider vectors $v^n \in V_\infty(\lambda-2n)$ for $n<0$. That is, the highest weight Verma module is extended to the principal series module
\[\cdots\overset{e}{\underset{f}{\rightleftharpoons}}V_\infty(\lambda-2)\overset{e}{\underset{f}{\rightleftharpoons}} V_\infty(\lambda)
\overset{e}{\underset{f}{\rightleftharpoons}}V_\infty(\lambda+2)\overset{e}{\underset{f}{\rightleftharpoons}}V_\infty(\lambda+4)\overset{e}{\underset{f}{\rightleftharpoons}}\cdots.\]
The new upper-half part of this module,
\[V_\infty(\lambda+2)\overset{e}{\underset{f}{\rightleftharpoons}}V_\infty(\lambda+4)\overset{e}{\underset{f}{\rightleftharpoons}}\cdots,\]
which can be seen as the quotient of the principal series by the highest weight module, is actually the lowest weight Verma module $V^l_{\infty,\log_q x+1}$ with the lowest weight $\lambda+2 = \log_q x+1$. 
To see this, observe that, for $n\geq 0$,
\begin{align*}
    e v^{-n-1} &= [-n-1]v^{-(n+1)-1},\\
    f v^{-n-1} &= [\lambda+n+1]v^{-(n-1)-1},\\
    q^{\frac{h}{2}}v^{-n-1} &= q^{\frac{(\lambda+2) + 2n}{2}}v^{-n-1}.
\end{align*}
If we rescale the basis and define
\[
v_n = \frac{[n]!}{\prod_{k=0}^{n-1}[\lambda+2+k]} v^{-n-1},
\]
then 
\begin{align*}
    e v_n &= [-(\lambda+2)-n]v_{n+1},\\
    f v_n &= [n]v_{n-1},\\
    q^{\frac{h}{2}}v_n &= q^{\frac{(\lambda+2) + 2n}{2}}v_n,
\end{align*}
which agrees with the choice of basis for the lowest weight Verma module used in \cite{Park}.

\subsection{Inverted state sum and homogeneous braid links}
Since we want to get a power series in $x^{-1}$ as a result of the state sum, from \eqref{eq:bigO} we see that we would like $j,j'$ for the positive crossings to be big, and we would like $j,j'$ for the negative crossings to be small. For this purpose, we can try using the 4 types of crossings in Fig. \ref{fig:4basicblocks} as the basic building blocks. 
\begin{figure}[ht]
    \centering
    \includegraphics[scale=0.6]{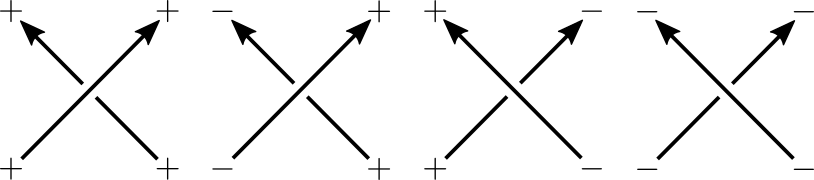}
    \caption{4 basic blocks}
    \label{fig:4basicblocks}
\end{figure}
In the figure, the signs denote whether we sum over non-negative integer weights ($+$ sign) or negative weights ($-$ sign). In other words, we are using the highest weight Verma module $V^h_{\infty,\log_q x - 1}$ for the strands labelled with $+$ and the lowest weight Verma module $V^l_{\infty, \log_q x +1}$ for the strands labelled with $-$. The extension of the $R$-matrix we described above gives the corresponding $R$-matrices. Thanks to the bound \eqref{eq:bigO} on the degree, any state sum built out of these 4 basic blocks always converges to a power series in $x^{-1}$ whose coefficients are Laurent polynomials in $q$. 
The class of links that can be made out of these basic blocks are known as the \emph{homogeneous braid links}. As a reminder, a braid is called homogeneous if, for each fixed position $k$, all the elementary braid $\sigma_k$ appears with either positive or negative powers in the braid word. A link is called a homogeneous braid link if it has a homogeneous braid presentation. 

Let us state in more precise terms what this new state sum model is. Given a homogeneous braid link presented as a closure of a homogeneous braid, we first open up the left-most strand, just as we did in \cite{Park}. Next, for each positive crossing we label the $j$ and $j'$ segments with $+$, and for each negative crossing we label the $j$ and $j'$ segments with $-$. (By a segment, we mean an edge of a braid or a knot diagram when we delete a small neighborhood of the set of all crossings.) In this way, all the columns except for the first column (the left-most one) are labeled with either $+$ or $-$. The first column can be labelled with either all $+$ or all $-$, and it doesn't matter. See Fig. \ref{fig:figure8braid} for an example of the figure-eight knot. In the example, we chose to label the first column with $+$ signs. 
\begin{figure}[ht]
    \centering
    \includegraphics[scale=0.6]{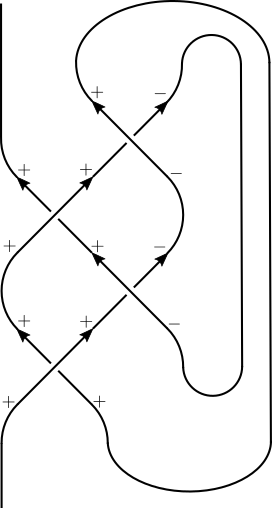}
    \caption{The figure-eight knot made out of the 4 basic blocks}
    \label{fig:figure8braid}
\end{figure}
Once we have labelled all the segments with either $+$ or $-$ sign so that all the crossings are one of the 4 basic blocks as in Fig. \ref{fig:4basicblocks}, we can take the partial quantum trace of the endomorphism on the tensor product of the highest and lowest weight Verma modules induced by the product of $R$-matrices. That is, we consider the state sum obtained by replacing all the crossings by the corresponding $R$-matrices and by taking the quantum trace. As a reminder, $v^n\in V_\infty(\log_q x-1-2n)$, so when a strand with weight $n$ gets closed up, we need to put an additional factor of $x^\frac{1}{2}q^{-\frac{1}{2}-n}$ when we take the quantum trace. As usual, we only sum over the indices of the internal segments and don't sum over the index of the open segment. Instead, we fix a vector in the highest weight Verma module if the open segment is labelled with $+$ and a vector in the lowest weight Verma module if it is labelled with $-$. It doesn't matter which vector we choose, but a typical choice that often simplifies the calculation is $v^0$ for the highest weight Verma module and $v^{-1}$ for the lowest weight Verma module. 
Let's denote the result of this partial quantum trace by $\sum(R\cdots R)$. 
\begin{defn}\label{def:invertedStateSum}
    Given a homogeneous braid diagram $\beta$, the \emph{inverted state sum} is
    \begin{equation}
        Z^{\textrm{inv}}(\beta) := (-1)^{s}\sum(R\cdots R),
    \end{equation}
    where 
    \[
    s = (\#\textrm{ of columns with negative crossings}) + (\#\textrm{ of }(\check{R}^{-1})_{-,-}^{-,-}\text{ blocks}).
    \]
\end{defn}
The role of the sign $(-1)^s$ will become clear below. Our main result of this section is that the inverted state sum gives an invariant of homogeneous braid links. 
\begin{thm}\label{thm:mainTheorem}
For any homogeneous braid link $L$ with a homogeneous braid diagram $\beta_L$, let
\[
F_L := (x^{\frac{1}{2}}-x^{-\frac{1}{2}})Z^{\textrm{inv}}(\beta_L),
\]
where $x$ is the parameter associated to the open strand. Then the followings are true. 
\begin{enumerate}
    \item $F_L$ is an invariant of $L$. That is, it is independent of the choice of the homogeneous braid representative. 
    \item Setting $q=e^{\hbar}$, its $\hbar$-expansion agrees with the Melin-Morton-Rozansky expansion of the colored Jones polynomials. 
    \item $F_L$ is annihilated  by the quantum $A$-polynomial (or quantum $A$-ideal in case it has multiple components). 
\end{enumerate}
In other words, $F_L$ is the invariant whose existence was conjectured in \cite{GM}. 
\end{thm}
\begin{proof}
For simplicity, let's focus on the case $L$ is a knot. The argument we are about to present can be easily generalized to the case of links. In this case, $Z^{\textrm{inv}}(\beta_L)$ is an element of $\mathbb{Z}[q,q^{-1}][[x^{-1}]]$. Our goal is to show that the $\hbar$-expansion of this series agrees with the Melvin-Morton-Rozansky expansion of the colored Jones polynomials of $L$ expanded near $x=\infty$, which is part (2) of the Theorem. Before doing that, let's first see how part (1) and (3) immediately follow from part (2). Part (1) follows from part (2) because the Melvin-Morton-Rozansky expansion is an invariant of a knot, and part (2) shows that it can be re-summed into a series in $x^{-1}$ with coefficients in $\mathbb{Z}[q,q^{-1}]$. Since the $\hbar$-series uniquely determines the Laurent polynomial in $q=e^{\hbar}$, $F_L$ itself is an invariant of $L$, and that proves part (1). Proof of part (3) from (2) is similar. Since the Melvin-Morton-Rozansky expansion is annihilated by the quantum $A$-polynomial, $\hat{A}_L Z^{\textrm{inv}}(\beta_L)$ should vanish when expanded into a series in $\hbar$. Since $\hat{A}_L Z^{\textrm{inv}}(\beta_L) \in \mathbb{Z}[q,q^{-1}][[x^{-1}]]$, it follows that each coefficient should vanish, meaning that $\hat{A}_L Z^{\textrm{inv}}(\beta_L) = 0$. 

Now let's prove part (2). For this, we combine ideas from \cite{LW} and \cite{R}. Following \cite{R}, we define the \emph{parametrized $R$-matrices} to be
\begin{align}\label{eq:parametrizedRmatrices}
    \mathsf{R}(\alpha,\beta,\gamma)_{i,j}^{i',j'} &= \delta_{i+j,i'+j'}\binom{i}{j'}\alpha^j \beta^{j'} \gamma^{i-j'}, \\
    \mathsf{R}^{-1}(\alpha,\beta,\gamma)_{i,j}^{i',j'} &= \delta_{i+j,i'+j'}\binom{j}{i'}\alpha^i \beta^{i'} \gamma^{j-i'},\nonumber
\end{align}
where $i,j,i',j' \geq 0$. 
Given an oriented knot diagram with one strand open, we can consider the state sum by replacing each positive crossing with $\mathsf{R}$ and negative crossing with $\mathsf{R}^{-1}$. We will use different parameters for each crossing, so there will be $3c$ independent parameters in total, where $c$ is the number of crossings of the knot diagram. It should be noted that the parametrized $R$-matrices do \emph{not} satisfy either Yang-Baxter relation or unitarity. In particular, $\mathsf{R}$ is \emph{not} the inverse matrix of $\mathsf{R}^{-1}$; we are abusing the notation for convenience. The significance of the parametrized $R$-matrix comes from the fact that they induce an algebra morphism in the following way. Denoting the vector $v^i\otimes v^j$ by a monomial $z_1^i z_2^j$, we see that
\begin{align*}
    \mathsf{R}(z_1^iz_2^j) &= (\gamma z_1 + \beta z_2)^i(\alpha z_1)^j = \mathsf{R}(z_1)^i\mathsf{R}(z_2)^j,\\
    \mathsf{R}^{-1}(z_1^iz_2^j) &= (\alpha z_2)^i(\beta z_1 + \gamma z_2)^j = \mathsf{R}^{-1}(z_1)^i\mathsf{R}^{-1}(z_2)^j.
\end{align*}
Put in a slightly different language, this means that the parametrized $R$-matrices induce a model of random walk of free bosons on the knot diagram. It follows from the theorem of Foata and Zeilberger \cite{FZ} (see also \cite{LW} for an exposition of Foata-Zeilberger formula) that the result of this state sum is
\begin{equation}
    \mathsf{Z}(\{\alpha\},\{\beta\},\{\gamma\}) = \frac{1}{\det(I - \mathcal{B})},
\end{equation}
where $\mathcal{B}$ is the transition matrix of this model of random walk. That is, $\mathcal{B}$ is the $n\times n$ matrix where $n$ is the number of internal segments of the knot diagram, and it records the probability (or the weight) of a boson to jump from one segment to another. The weight is determined by the corresponding entry of the matrix $\mathsf{R}$ or $\mathsf{R}^{-1}$, with $(i,j) = (0,1)$ or $(1,0)$ and $(i',j') = (0,1)$ or $(1,0)$. From the definition of the transition matrix, it is easy to see that the denominator, $\det(I - \mathcal{B})$, can be expressed in the following way. 
\begin{equation}\label{eq:detAsSumOfCycles}
        \det(I - \mathcal{B}) = \sum_{c}(-1)^{|c|}W(c),
\end{equation}
where the sum ranges over all simple multi-cycles $c$, $|c|$ is the number of components of the multi-cycle $c$, and $W(c)$ denotes the weight of $c$. Let us clarify some of the terminologies. A \emph{cycle} is a cyclic path (without a starting point or an end point) on the knot diagram as an oriented graph. A \emph{multi-cycle} is an unordered tuple of cycles. We call a multi-cycle $\emph{simple}$ if it uses each segment at most once. Since it counts simple multi-cycles (with weights), $\det(I-\mathcal{B})$ can be obtained as a result of a state sum where the space of spin states is spanned by $0$ and $1$, instead of all non-negative integers, as long as we keep track of the sign $(-1)^{|c|}$. In this sense, $\det(I - \mathcal{B})$ can be obtained as a state sum in a system of random walk of free fermions. 

Now suppose that the oriented knot diagram is a homogeneous braid diagram, and let's `invert' this model of random walk. That is, we consider the inverted state sum, but using $\mathsf{R}$ and $\mathsf{R}^{-1}$ instead $\check{R}$ and $\check{R}^{-1}$. Explicitly, the 4 basic blocks are given by
\begin{align}\label{eq:invertedParametrizedRmatrices}
    \mathsf{R}_{+,+}^{+,+} &: \delta_{i+j,i'+j'}\binom{i}{j'}\alpha^j \beta^{j'} \gamma^{i-j'},\\
    \mathsf{R}_{-,+}^{-,+} &: \delta_{i+j,i'+j'}\gamma^{-1}\binom{j'+(-i-1)}{j'}\alpha^j (-\beta)^{j'} \gamma^{-(-i-1)-j'},\nonumber\\
    (\mathsf{R}^{-1})_{+,-}^{+,-} &: \delta_{i+j,i'+j'}\gamma^{-1}\binom{i'+(-j-1)}{i'}\alpha^i(-\beta)^{i'}\gamma^{-i'-(-j-1)},\nonumber\\
    (\mathsf{R}^{-1})_{-,-}^{-,-} &: \delta_{i+j,i'+j'} \alpha^{-1}\beta^{-1}\binom{-i'-1}{-j-1}\alpha^{-(-i-1)}\beta^{-(-i'-1)}(-\gamma)^{(-i'-1)-(-j-1)}.\nonumber
\end{align}
Extracting out the factors $\gamma^{-1}$ and $\alpha^{-1}\beta^{-1}$ to normalize these `inverted' parametrized $R$-matrices, we see that this again gives a model of random walk of free bosons. This can be best described with what we call ``highway diagrams". Before inversion, the highway diagrams for the positive and the negative crossing look like Fig. \ref{fig:highwayDiagrams}. 
\begin{figure}[ht]
    \centering
    \includegraphics[scale=0.6]{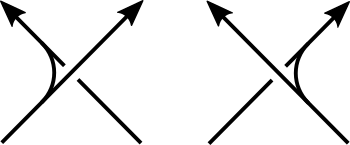}
    \caption{Highway diagrams for the positive and the negative crossing}
    \label{fig:highwayDiagrams}
\end{figure}
Note the arcs connecting the over-strand with the under-strand. They denote the way the bosons can move; they can move from the over-strand to the under-strand, but not the other way around. 
After inverting some of the indices according to the 4 basic blocks as in Fig. \ref{fig:4basicblocks}, we can make change of variables to the inverted indices and use $k_{\mathrm{inv}} := -k-1 \in \mathbb{Z}_{\geq 0}$ instead of $k\in \mathbb{Z}_{<0}$ for each inverted index $k$ ($i$, $j$, $i'$, or $j'$). The effect of this change of variables to the diagram is to invert the orientation of the inverted segments. As a result, the inverted highway diagrams for the 4 basic blocks look like Fig. \ref{fig:invertedHighwayDiagrams}. 
\begin{figure}[ht]
    \centering
    \includegraphics[scale=0.6]{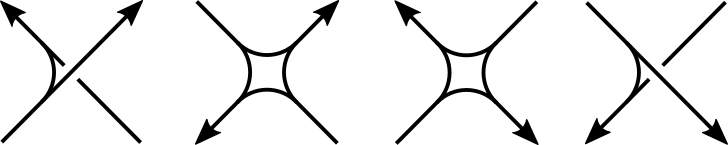}
    \caption{Inverted highway diagrams for the 4 basic blocks}
    \label{fig:invertedHighwayDiagrams}
\end{figure}
The inverted parametrized $R$-matrices can now be thought of as the maps from the incoming strands to the outgoing strands. After extracting out the pre-factor ($\gamma^{-1}$ for $\mathsf{R}_{-,+}^{-,+}$ and $(\mathsf{R}^{-1})_{+,-}^{+,-}$, and $\alpha^{-1}\beta^{-1}$ for $(\mathsf{R}^{-1})_{-,-}^{-,-}$), each of these maps induce an algebra morphism in the sense we discussed above. For instance, in case of $\gamma \mathsf{R}_{-.+}^{-,+}$, denoting the vector $v_{i'_{\mathrm{inv}}} \otimes v^j$ by a monomial $z_1^{i'_{\mathrm{inv}}}z_2^{j}$, we see that
\[
\gamma \mathsf{R}(z_1^{i'_{\mathrm{inv}}}z_2^{j}) = (\gamma^{-1}z_1 - \beta\gamma^{-1}z_2)^{i'_{\mathrm{inv}}}(\alpha\gamma^{-1}z_1 - \alpha\beta\gamma^{-1}z_2)^j = \qty(\gamma \mathsf{R}(z_1))^{i'_{\mathrm{inv}}}\qty(\gamma \mathsf{R}(z_2))^j.
\]
It follows that the result of the state sum of this new model of random walk of free bosons is
\begin{equation}\label{eq:parametrizedInvertedStateSum}
    \mathsf{Z}^{\mathrm{inv}}(\{\alpha\},\{\beta\},\{\gamma\}) = \frac{\prod_{c_2}\gamma^{-1}\prod_{c_3}\gamma^{-1}\prod_{c_4}\alpha^{-1}\beta^{-1}}{\det(I - \mathcal{B}_{\mathrm{inv}})},
\end{equation}
where $\mathcal{B}_{\mathrm{inv}}$ is the transition matrix of this new model of random walk (determined by $\mathsf{R}_{+,+}^{+,+}, \gamma\mathsf{R}_{-,+}^{-,+}, \gamma\mathsf{R}_{+,-}^{+,-}, \alpha\beta\mathsf{R}_{-,-}^{-,-}$), and the term in the numerator denotes the product of $\gamma^{-1}$ for each crossing of the second and the third type multiplied by the product of $\alpha^{-1}\beta^{-1}$ for each crossing of the fourth type, basically pulling out the pre-factor. 

We will prove the following lemma. 
\begin{lem}\label{lem:parametrizedInvertedStateSumFormula}
\begin{equation}
    \mathsf{Z}(\{\alpha\},\{\beta\},\{\gamma\}) = (-1)^s \mathsf{Z}^{\mathrm{inv}}(\{\alpha\},\{\beta\},\{\gamma\}),
\end{equation}
where $s$ is as in Definition \ref{def:invertedStateSum}. 
\end{lem}
To prove the lemma, let's compare the denominators, $\det(I-\mathcal{B})$ and $\det(I-\mathcal{B}_{\mathrm{inv}})$. Recall from \eqref{eq:detAsSumOfCycles} that each of these determinants can be understood as the weighted sum over all simple multi-cycles. So our strategy is to find a one-to-one correspondence between the set of all simple multi-cycles in the first model of random walk of free fermions and that of the second model. 
A crucial observation is the following correspondence. 
\begin{prop}\label{prop:inversionCorrespondence}
For $i,i',j,j'\in \{0,1\}$, there is a correspondence between the parametrized $R$-matrix and its inverted version , as in Fig. \ref{fig:inversionCorrespondence}. 
\begin{figure}[ht]
    \centering
    \includegraphics[scale=0.4]{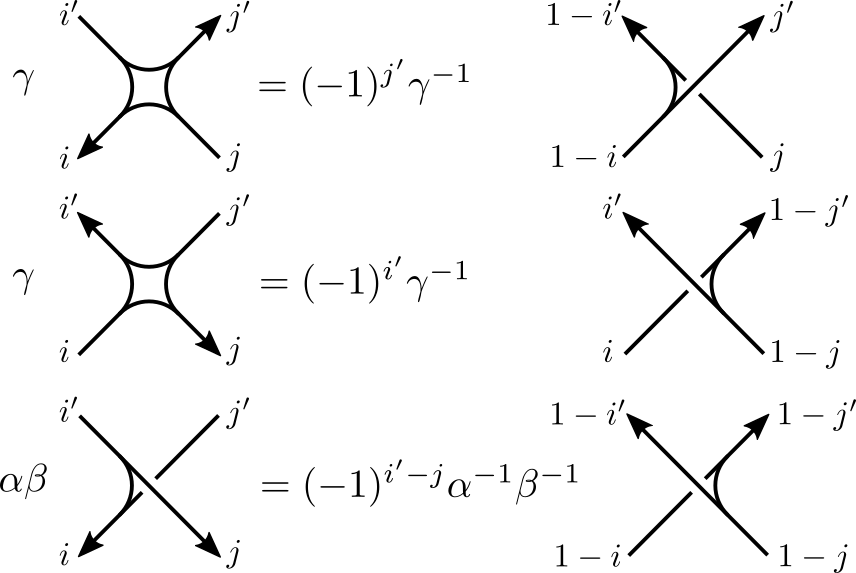}
    \caption{The correspondence between the parametrized $R$-matrix and its inversion}
    \label{fig:inversionCorrespondence}
\end{figure}
\end{prop}
This can be proved by checking all the cases. There are two cases we need to be a little careful of. Those are when $i,j,i',j'$ are all $1$ in the crossing of the type either $\mathsf{R}_{-,+}^{-,+}$ or $\mathsf{R}_{+,-}^{+,-}$. In those cases, the right-hand side of the equations in Fig. \ref{fig:inversionCorrespondence} are zero, but naively the inverted parametrized $R$-matrices \eqref{eq:invertedParametrizedRmatrices} themselves are non-zero. This is okay, and in fact when we consider the fermionic model that computes $\det(I-\mathcal{B}_{\mathrm{inv}})$ (instead of the bosonic model that computes its inverse), it is more natural to set the value of the matrices $\mathsf{R}_{-,+}^{-,+}$ and $\mathsf{R}_{+,-}^{+,-}$ to be 0 when $i,j,i',j'$ are all $1$. This is because in those types of crossing, when $i,j,i',j'$ are all $1$, then there are two different ways to make it into a simple multi-cycle. That is, there are two different simple multi-cycles realizing that configuration. For instance, in case of $\mathsf{R}_{-,+}^{-,+}$, we can either connect $i'$ to $j'$ and $j$ to $i$, or we can connect $i'$ to $i$ and $j$ to $j'$. Moreover, one of the two simple multi-cycles have one more component than the other. Since we are counting simple multi-cycles with weight and sign as in \ref{eq:detAsSumOfCycles}, the contribution of that configuration is 0. Therefore in this fermionic model $\mathsf{R}_{-,+}^{-,+}$ and $\mathsf{R}_{+,-}^{+,-}$ are zero when $i,j,i',j'$ are all $1$, and this proves the proposition. 

We are now ready to compare $\det(I-\mathcal{B})$ with $\det(I-\mathcal{B}_{\mathrm{inv}})$. Let $D$ and $D_{\mathrm{inv}}$ be the highway diagrams for the original knot diagram and its inversion. Any multi-cycles in either $D$ or $D_{\mathrm{inv}}$ can be thought of as a configuration of $0$'s and $1$'s on the diagram. Moreover, any such configuration that has a non-zero weight uniquely determines the simple multi-cycle. 
Following the rule as in Fig. \ref{fig:inversionCorrespondence}, for each simple multi-cycle $c$ in $D$, there is a corresponding multi-cycle $c_{\mathrm{inv}}$ in $D_{\mathrm{inv}}$ and vice versa. Even their weights agree, up to a simple factor as given in Fig. \ref{fig:inversionCorrespondence}. We need to carefully study these extra factors. 
The factors $\gamma^{-1}$ and $\alpha^{-1}\beta^{-1}$ give an overall factor of $\prod_{c_2}\gamma^{-1}\prod_{c_3}\gamma^{-1}\prod_{c_4}\alpha^{-1}\beta^{-1}$, which gets cancelled out with the numerator of \eqref{eq:parametrizedInvertedStateSum}. 

Let's take a careful look at the sign factors. We claim that the overall sign factor agrees with $(-1)^{|c| - |c_{\mathrm{inv}}| + s}$, where $s$ is as in Definition \ref{def:invertedStateSum}. The basic idea is the following. Suppose we have $n$ number of disjoint intervals on a circle, and suppose that the ends of the intervals are connected in a certain way so that they form a collection of cycles. If we change the set of intervals on the circle to their complement, then the resulting number of cycles and the original number of cycles have the same parity if $n$ is odd, and they have the opposite parity if $n$ is even. See Fig. \ref{fig:numberOfComponents} as an illustration. 
\begin{figure}[ht]
    \centering
    \includegraphics[scale=0.4]{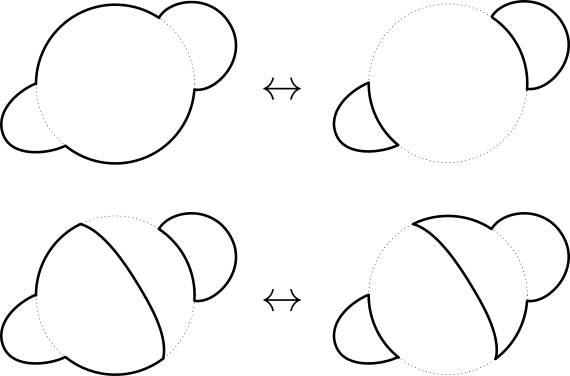}
    \caption{The change of parity of the number of components as we invert the circle}
    \label{fig:numberOfComponents}
\end{figure}
This is exactly what happens when we invert a column of a knot diagram. Under the correspondence in Fig. \ref{fig:inversionCorrespondence}, $0$'s and $1$'s in a column labelled with $-$ signs change place. So, for instance, if we have a column labelled with $-$ signs and if the adjacent columns are labelled with $+$ signs, then a typical picture would be something like Fig. \ref{fig:inversionNnumberOfComponents}. 
\begin{figure}[ht]
    \centering
    \includegraphics[scale=0.6]{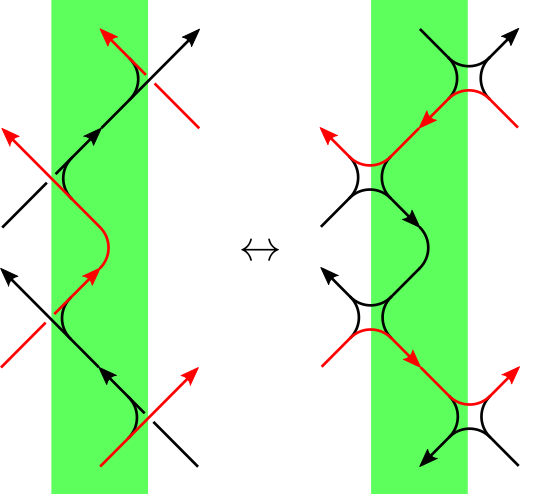}
    \caption{Inverting a column labelled with $-$ sign}
    \label{fig:inversionNnumberOfComponents}
\end{figure}
In the figure, the column we are inverting is highlighted with a green stripe, and this is the one that plays the role of the circle in Fig. \ref{fig:numberOfComponents}. A part of a typical multi-cycle is drawn with red lines. As pointed out above, the change of parity of the number of components of the multi-cycle is determined by the parity of the red intervals in the green stripe. The number of such intervals equals the number of red outgoing strands from the green stripe, which is the sum of $i'$'s for the crossings on the left edge of the stripe plus the sum of $j'$'s for the crossings on the right edge of the stripe. 

More generally, there can be some number of columns labelled with $-$ signs that are adjacent to each other. In that case, the number of circles is not just the number of columns that we invert, because some of the circles are connected through crossings of type $(\mathsf{R}^{-1})^{-,-}_{-,-}$ with either $(i,j,i',j') = (1,0,0,1),\, (0,1,1,0)$ or $(1,1,1,1)$. As a result, the parity of the number of circles that we invert has the same parity with the number of columns with $-$ sign plus the number of such crossings. Note that the number of crossings of type $(\mathsf{R}^{-1})^{-,-}_{-,-}$ with either $(i,j,i',j') = (1,0,0,1),\, (0,1,1,0)$ or $(1,1,1,1)$ equals the number of all crossings of type $(\mathsf{R}^{-1})^{-,-}_{-,-}$ minus the number of such crossings with $(i,j,i',j') = (0,1,0,1)$ (or $(1,0,1,0)$ after inversion). 
All in all, we have
\[
(-1)^{\textrm{change of parity in the }\#\textrm{ of components}} = (-1)^{\sum_{c_2}{j'}+\sum_{c_3}{i'}+\sum_{c_4}{(i'-j)}+s},
\]
where $\sum_{c_2}{j'}+\sum_{c_3}{i'}+\sum_{c_4}{(i'-j)}$ denotes the sum of $j'$'s for all the crossings of the second type plus the sum of $i'$'s for all the crossings of the third type plus the sum of $(i'-j)$'s for all the crossings of the fourth type. Comparing with Fig. \ref{fig:inversionCorrespondence}, we see that the overall change of sign in the weight of any multi-cycle is just $(-1)^s$. This concludes the proof of Lemma \ref{lem:parametrizedInvertedStateSumFormula}. 

In the classical limit, the $R$-matrices $\check{R}(x)$ and $\check{R}^{-1}(x)$ can be obtained by specializing the parameters of the parametrized $R$-matrix. More precisely, 
\begin{align}\label{eq:specializationOfParametrizedRmatrices}
    \check{R}(x)_{i,j}^{i',j'} \bigg\vert_{q=1} &= \mathsf{R}(\alpha,\beta,\gamma)_{i,j}^{i',j'} \bigg\vert_{\alpha = x^{-\frac{1}{2}}, \beta = x^{-\frac{1}{2}}, \gamma = 1-x^{-1}},\\
    \check{R}^{-1}(x)_{i,j}^{i',j'} \bigg\vert_{q=1} &= \mathsf{R}^{-1}(\alpha,\beta,\gamma)_{i,j}^{i',j'} \bigg\vert_{\alpha = x^{\frac{1}{2}}, \beta = x^{\frac{1}{2}}, \gamma = 1-x}.\nonumber
\end{align}
This fact, combined with Lemma \ref{lem:parametrizedInvertedStateSumFormula}, immediately proves that the classical limit of $Z^{\mathrm{inv}}(\beta_L)$ agrees with $\frac{1}{\Delta_L(x)}$, where $\Delta_L(x)$ denotes the Alexander polynomial of $L$. The rest of the proof follows easily from what Rozansky proved in \cite{R}; there exists a differential operator $D_n$ in the parameters $\{\alpha\}, \{\beta\}, \{\gamma\}$ with polynomial coefficients such that the coefficient of $\hbar^n$ in the Melvin-Morton-Rozansky expansion is
\begin{equation}\label{eq:higherhbarTermBeforeInversion}
D_n\,\mathsf{Z}(\{\alpha\},\{\beta\},\{\gamma\}) \bigg\vert_{\textrm{specialize }\{\alpha\}, \{\beta\}, \{\gamma\}\textrm{ according to \eqref{eq:specializationOfParametrizedRmatrices}}}.
\end{equation}
The way it is proved is by observing that the $\hbar$-expansion of $\check{R}(x)$ can be obtained by acting a certain differential operator in $\alpha, \beta, \gamma$ to $\mathsf{R}(\alpha, \beta, \gamma)$ and then specializing these parameters. 
It is easy to see that the differential operator for the $R$-matrix remains the same even when we invert some of the indices, and it follows that the operator $D_n$ itself remains the same under inversion. So the corresponding $\hbar^n$-coefficient for the inverted state sum is given by
\begin{equation}\label{eq:higherhbarTermAfterInversion}
(-1)^s D_n\,\mathsf{Z}^{\mathrm{inv}}(\{\alpha\},\{\beta\},\{\gamma\})\bigg\vert_{\textrm{specialize }\{\alpha\}, \{\beta\}, \{\gamma\}\textrm{ according to \eqref{eq:specializationOfParametrizedRmatrices}}}.
\end{equation}
By Lemma \ref{lem:parametrizedInvertedStateSumFormula}, \eqref{eq:higherhbarTermBeforeInversion} and \eqref{eq:higherhbarTermAfterInversion} are the same as rational functions. This finishes the proof of Theorem \ref{thm:mainTheorem}. 
\end{proof}

\begin{rmk}
It is important to note that \eqref{eq:higherhbarTermBeforeInversion} is a power series in $(1-x)$ while \eqref{eq:higherhbarTermAfterInversion} is a power series in $x^{-1}$. In other words, a state sum and its inversion are typically defined in a different domain; one near $x=1$ and the other near $x=\infty$ (or $x=0$ if we had started with the lowest weight Verma module). For this reason, they can't be compared directly. This is why even though a lot is known about the integral form of the Melvin-Morton-Rozansky expansion cyclotomically (see e.g. \cite{Hab, Hab2, Willetts}), it has been a challenge to find its integral form near $x=0$ or $x=\infty$ as \cite{GM} conjectured. 

In the above proof, we circumvented the issue of having different domains of convergence by making use of the fact that they can both be expressed as rational functions and showed that the two rational functions are the same. 
\end{rmk}

\begin{rmk}
Let $L$ be a homogeneous braid link. 
In case the link has $l$ number of components, let $x_1, \cdots, x_l$ be the corresponding $x$-variables. Define $\deg_x x_n = 1$ for each $1\leq n\leq l$. Then analogously to \eqref{eq:bigO}, 
\begin{align}
    \check{R}(x_n,x_m)_{i,j}^{i',j'} &= O(x^{-\frac{j+j'+1}{2}}),\\
    \check{R}^{-1}(x_n,x_m)_{i,j}^{i',j'} &= O(x^{\frac{j+j'+1}{2}}),\nonumber
\end{align}
as a power series in $x^{-1}$, for any $1\leq n,m\leq l$. From this bound on the degree, it is easy to see that the inverted state sum for $L$ is a well-defined power series whose coefficients are Laurent polynomials. 

By the theorem of Stallings \cite{Stallings}, $L$ is fibered. In particular, its Alexander-Conway polynomial is non-zero, and therefore the Melvin-Morton-Rozansky expansion is well-defined. Theorem \ref{thm:mainTheorem} says that, when written as a power series in $x^{-1}$, its coefficients can be re-summed into Laurent polynomials in $q$. 
\end{rmk}

\begin{rmk}
The effect of the sign factor on the positive and negative stabilization moves can be summarized as in Fig. \ref{fig:stabilizations}. 
\begin{figure}[ht]
    \centering
    \includegraphics[scale=0.7]{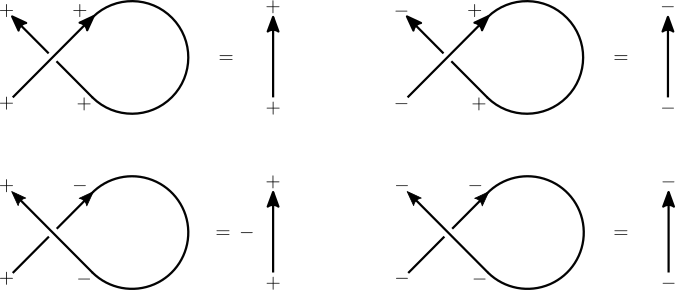}
    \caption{Sign factors under stabilizations}
    \label{fig:stabilizations}
\end{figure}
\end{rmk}

\begin{rmk}
When specialized according to \eqref{eq:specializationOfParametrizedRmatrices}, the equation \eqref{eq:detAsSumOfCycles} becomes Murakami's state sum expression for the (multi-variable) Alexander polynomial \cite{M}. In Murakami's state sum model, the space of spin states is a 2-dimensional super-vector space $V = \mathrm{span}\{v_0,v_1\}$ where $v_0$ has even degree and $v_1$ has odd degree. The trace is replaced by the super-trace, and one of the $R$-matrix element has an extra $-$ sign. The set of $v_1$-colored segments form a simple multi-cycle, and the number of components is counted by the signs coming from the super-trace and the $R$-matrix element. 
\end{rmk}

\begin{eg}[Figure-eight knot]
Take the braid $\sigma_1\sigma_2^{-1}\sigma_1\sigma_2^{-1}$ from bottom to top. We close up the second and the third strand while leaving the first strand open. 
Since this braid is homogeneous, it can be built out of the 4 basic blocks. See Fig. \ref{fig:figure8braid}. 
Before closing up the braid, let's say the indices of the three open strands are $0, m, k$ from left to right. In this case, these three indices uniquely determine all the internal indices. Summing over $m\geq 0$ and $k<0$, we get the following expression for $F_{\mathbf{4}_1}$:
\begin{align*}
F_{\mathbf{4}_1}(x,q) &= -(x^{\frac{1}{2}}-x^{-\frac{1}{2}})\sum_{\substack{m\geq 0\\ k< 0}}\check{R}(x)_{0,m}^{m,0}\,\check{R}^{-1}(x)_{0,k}^{0,k}\,\check{R}(x)_{m,0}^{0,m}\,\check{R}^{-1}(x)_{m,k}^{m,k}\cdot x^{\frac{1}{2}}q^{-\frac{1}{2}-m} \cdot x^{\frac{1}{2}}q^{-\frac{1}{2}-k}\\
&= -x^{-\frac{1}{2}} -2 x^{-\frac{3}{2}} -(\frac{1}{q}+3+q) x^{-\frac{5}{2}} - (\frac{2}{q^2}+\frac{2}{q}+5+2q+2q^2)x^{-\frac{7}{2}} + O(x^{-\frac{9}{2}}).
\end{align*}
In \cite{GM} this series was computed term by term using recursion, but we have found a closed formula! 
\end{eg}

\begin{eg}[$6_2$ knot]\label{eg:62knot}
Take the braid $\sigma_1^3\sigma_2^{-1}\sigma_1\sigma_2^{-1}$ from bottom to top. Computing the inverted state sum, we get
\[
F_{\mathbf{6}_2}(x,q) = -q x^{-\frac{3}{2}} -2q x^{-\frac{5}{2}} +(-1-3q+q^2) x^{-\frac{7}{2}} +(-\frac{2}{q}-2-4q+2q^2)x^{-\frac{9}{2}} +O(x^{-\frac{11}{2}}).
\]
\end{eg}

\begin{eg}[$6_3$ knot]\label{eg:63knot}
Take the braid $\sigma_1^2\sigma_2^{-1}\sigma_1\sigma_2^{-2}$ from bottom to top. Computing the inverted state sum, we get
\[
F_{\mathbf{6}_3}(x,q) = x^{-\frac{3}{2}} +2x^{-\frac{5}{2}} +(-\frac{1}{q}+3-q)x^{-\frac{7}{2}} +(-\frac{2}{q^2}-\frac{2}{q}+4-2q-2q^2)x^{-\frac{9}{2}} +O(x^{-\frac{11}{2}}).
\]
Note that each coefficient has $q\leftrightarrow q^{-1}$ symmetry, which is due to amphichirality of $\mathbf{6}_3$. 
\end{eg}

\begin{eg}[Whitehead link]\label{eg:Whitehead}
Take the braid $\sigma_1\sigma_2^{-1}\sigma_1\sigma_2^{-1}\sigma_1$ from bottom to top. Computing the inverted state sum, we get
\[
F_{\mathbf{Wh}}(x_1,x_2,q) = -q^{\frac{1}{2}}\sum_{n,m\geq 0}f_{n,m}^{\mathbf{Wh}}x_1^{-n-\frac{1}{2}}x_2^{-m-\frac{1}{2}},
\]
where
\[
f^{\mathbf{Wh}} = 
\begin{psmallmatrix}
1 & 1 & 1 & 1 &\cdots\\
1 & \frac{1}{q}+1-q & \frac{1}{q^2}+\frac{1}{q}+1-q-q^2 & \frac{1}{q^3}+\frac{1}{q^2}+\frac{1}{q}+1-q-q^2-q^3 & \cdots\\
1 & \frac{1}{q^2}+\frac{1}{q}+1-q-q^2 & \frac{1}{q^4}+\frac{1}{q^3}+\frac{2}{q^2}+\frac{1}{q} -2q -2q^2 & \frac{1}{q^6}+\frac{1}{q^5}+\frac{2}{q^4}+\frac{2}{q^3}+\frac{2}{q^2}-1-3q-3q^2-q^3+q^5 & \cdots\\
1 & \frac{1}{q^3}+\frac{1}{q^2}+\frac{1}{q}+1-q-q^2-q^3 & \frac{1}{q^6}+\frac{1}{q^5}+\frac{2}{q^4}+\frac{2}{q^3}+\frac{2}{q^2}-1-3q-3q^2-q^3+q^5 & \substack{\frac{1}{q^9}+\frac{1}{q^8}+\frac{2}{q^7}+\frac{3}{q^6}+\frac{3}{q^5}+\frac{3}{q^4}+\frac{2}{q^3}\\-\frac{3}{q}-4-6q-4q^2-2q^3+q^4+2q^5+q^6+q^7} & \cdots\\
\vdots & \vdots & \vdots & \vdots & \ddots
\end{psmallmatrix}.
\]
This matches perfectly with the computations done in \cite{Park} (one is mirror to another). 
\end{eg}

\begin{eg}[Borromean rings]\label{eg:Borromean}
Take the braid $\sigma_1\sigma_2^{-1}\sigma_1\sigma_2^{-1}\sigma_1\sigma_2^{-1}$ from bottom to top. Computing the inverted state sum, we get
\[
F_{\mathbf{Bor}}(x_1,x_2,x_3,q) = \sum_{n,m,l\geq 0}f_{n,m,l}^{\mathbf{Bor}}x_1^{-n-\frac{1}{2}}x_2^{-m-\frac{1}{2}}x_3^{-l-\frac{1}{2}},
\]
where
\[
f_0^{\mathbf{Bor}} =
\begin{psmallmatrix}
1 & 1 & 1 & \cdots\\
1 & 1 & 1 & \cdots\\
1 & 1 & 1 & \cdots\\
\vdots & \vdots & \vdots & \ddots
\end{psmallmatrix},
\]
\[
f_1^{\mathbf{Bor}} =
\begin{psmallmatrix}
1 & 1 & 1 & \cdots\\
1 & -\frac{1}{q^2}+3-q^2 & -\frac{1}{q^3}-\frac{1}{q^2}+\frac{1}{q}+3+q-q^2-q^3 & \cdots\\
1 & -\frac{1}{q^3}-\frac{1}{q^2}+\frac{1}{q}+3+q-q^2-q^3 & -\frac{1}{q^4}-\frac{2}{q^3}-\frac{1}{q^2}+\frac{2}{q}+5+2q-q^2-2q^3-q^4 & \cdots\\
\vdots & \vdots & \vdots & \ddots
\end{psmallmatrix},
\]
\[
f_2^{\mathbf{Bor}} =
\begin{psmallmatrix}
1 & 1 & 1 & \cdots\\
1 & -\frac{1}{q^3}-\frac{1}{q^2}+\frac{1}{q}+3+q-q^2-q^3 & -\frac{1}{q^4}-\frac{2}{q^3}-\frac{1}{q^2}+\frac{2}{q}+5+2q-q^2-2q^3-q^4 & \cdots\\
1 & -\frac{1}{q^4}-\frac{2}{q^3}-\frac{1}{q^2}+\frac{2}{q}+5+2q-q^2-2q^3-q^4 & \frac{1}{q^7}-\frac{1}{q^5}-\frac{5}{q^4}-\frac{6}{q^3}-\frac{1}{q^2}+\frac{6}{q}+13+6q-q^2-6q^3-5q^4-q^5+q^7 & \cdots\\
\vdots & \vdots & \vdots & \ddots
\end{psmallmatrix},
\]
and so on. Again, this is consistent with the computations done in \cite{Park}. Note that the coefficients have $q\leftrightarrow q^{-1}$ symmetry due to amphichirality of the Borromean rings. 
\end{eg}

\begin{eg}[$\mathbf{L7a1}$]\label{eg:L7a1}
Take the braid $\sigma_1^2\sigma_2^{-1}\sigma_1\sigma_2^{-1}\sigma_1\sigma_2^{-1}$ from bottom to top. Among the three strands that are open before closing up, let's say $x_1$ is the variable associated with the first two strands (they are connected once we close up the braid) and $x_2$ is the variable associated with the third strand. Computing the inverted state sum, we get
\[
F_{\mathbf{L7a1}}(x_1,x_2,q) = q^{\frac{1}{2}}\sum_{n,m\geq 0}f_{n,m}^{\mathbf{L7a1}}x_1^{-n-\frac{1}{2}}x_2^{-m-\frac{1}{2}},
\]
where
\[
f^{\mathbf{L7a1}} = 
\begin{psmallmatrix}
0 & 0 & 0 & 0 &\cdots\\
1 & 1 & 1 & 1 &\cdots\\
2 & 2 & 2 & 2 &\cdots\\
3-q & -\frac{1}{q^2}+4-q & -\frac{1}{q^3}-\frac{1}{q^2}+4 & -\frac{1}{q^4}-\frac{1}{q^3}-\frac{1}{q^2}+4+q^2 &\cdots\\
\vdots & \vdots & \vdots & \vdots & \ddots
\end{psmallmatrix}.
\]
The link $\mathbf{L7a1}$ consists of two unknots linked together with linking number $0$. See Fig. \ref{fig:L4a1}.
\begin{figure}[ht]
    \centering
    \includegraphics[scale=0.6]{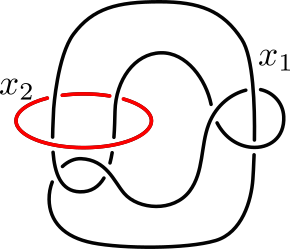}
    \caption{The link $\mathbf{L4a1}$}
    \label{fig:L4a1}
\end{figure}
So by doing $\frac{1}{r}$-surgery on one of the components, we can obtain various knots. For instance, $-1$, $+1$, $-\frac{1}{2}$. and $+\frac{1}{2}$-surgery on the $x_2$-component gives $\mathbf{6_3}$, $\mathbf{6_2}$, $\mathbf{8}_8$, and $\mathbf{8}_6$, respectively. These are good consistency checks, and indeed applying the partial $+1$ and $-1$ surgery formula (from \cite{Park} and also from Section \ref{sec:invertedHabiro} of this paper), we get the same result as in Examples \ref{eg:62knot} and \ref{eg:63knot}. 
\end{eg}

\subsection{Homogenization of braids}
As shown in \cite{Stallings}, for any link $L$ and an integer $n$, we can add an additional unknot component $O$, so that $L' = L \cup O$ is a homogeneous braid link, and $\mathrm{lk}(L,O) = n$. 
It follows that the problem of defining $F_L$ for all links is reduced to the following problem of finding an $\infty$-surgery formula. 
\begin{qn}\label{qn:infinitySurgery}
    Let $L'$ be a link obtained by adding a new component $K$ to a link $L$. Is there a formula relating $F_{L'}$ with $F_L$?
\end{qn}
\begin{rmk}
In \cite{Turaev}, there is a nice $\infty$-surgery formula for the torsion, or equivalently the Alexander-Conway polynomial in our case of link complements, when the $\mathrm{lk}(L,K)$ is non-zero. Written in terms of the inverse torsions, the $\infty$-surgery formula is given by
\begin{equation}\label{eq:torsionInftySurgery}
\frac{1}{\tau(S^3\setminus L)} = \frac{[K]-1}{\textrm{in}(\tau(S^3\setminus L'))},
\end{equation}
where $[K]$ is the homology class of $K$ in $S^3\setminus L$, and $\mathrm{in}:\mathbb{Z}[H_1(S^3\setminus L')] \rightarrow \mathbb{Z}[H_1(S^3\setminus L)]$ is the inclusion homomorphism. 
So, Question \ref{qn:infinitySurgery} is asking if there is a $q$-deformation of the equation \eqref{eq:torsionInftySurgery}. 
\end{rmk}

If a link $L$ can be obtained by performing $-\frac{1}{r}$-surgery on an unknot component of a link $L'$, and if $L'$ is a homogeneous braid link, then we can compute $F_L$ by using the partial surgery formula \cite{Park}. This has been a useful strategy in computing $F_K$ for a variety of knots; for instance, double twist knots can be obtained by $\frac{1}{m}$, $\frac{1}{n}$ surgery on the Borromean rings, which is a homogeneous braid link.

\subsection{Beyond homogeneous braids}\label{subsec:BeyondHomogeneousBraids}
It is clear from the proof of Theorem \ref{thm:mainTheorem} that we don't actually need to invert along the columns. The only place we used homogeneity of the braid was to ensure that the inverted state sum converges absolutely, in a sense that for any $M>0$ there are only finitely many configurations that contribute any terms with $x^{-1}$-degree $<M$. 
As long as the resulting state sum converges absolutely, we can invert along any simple multi-cycle that follows the direction of the braid. More precisely, we have the following theorem. 
\begin{thm}\label{thm:beyondHomogeneousBraids}
Let $\beta$ be any braid. Suppose there is an assignment of either $+$ sign or $-$ sign to each segment of the braid closure in such a way that it satisfies the following two conditions. 
\begin{enumerate}
    \item[(a)] For each crossing, the number of incoming $-$-signed segments equals the number of outgoing $-$-signed segments. 
    \item[(b)] The partial quantum trace computed after inverting the $-$-signed segments converges absolutely. 
\end{enumerate}
Then, denoting the quantum trace by $\sum(R\cdots R)$ as before, we can define the corresponding inverted state sum by
\begin{equation}
    Z^{\mathrm{inv}}(\beta) := (-1)^s\sum(R\cdots R),
\end{equation}
where $s$ is the number of closed components of the set of $-$-signed segments as a simple multi-cycle\footnote{When two $-$-signed strands cross each other at a crossing, assume that the multi-cycle follows along the same strand.}, and the corresponding analogue of Theorem \ref{thm:mainTheorem} is true. 
\end{thm}
\begin{proof}
Most part of the proof of Theorem \ref{thm:mainTheorem} carries over immediately. The only thing we need to check is an analogue of Proposition \ref{prop:inversionCorrespondence} for each type of crossings that satisfy the condition (a) and to keep track of the sign factors carefully. Since it is straightforward to mimic what we did in the proof of Lemma \ref{lem:parametrizedInvertedStateSumFormula}, we will show this just for the two types of crossings shown in Fig. \ref{fig:2newcrossings}. Their highway diagrams are depicted in Fig. \ref{fig:2newcrossingsHighwayDiagrams}. 
\begin{figure}[ht]
    \centering
    \includegraphics[scale=0.5]{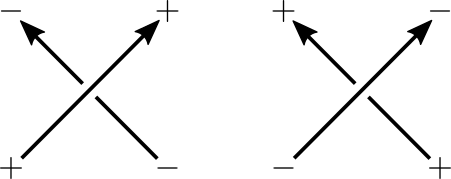}
    \caption{Two examples of new crossings}
    \label{fig:2newcrossings}
\end{figure}
\begin{figure}[ht]
    \centering
    \includegraphics[scale=0.55]{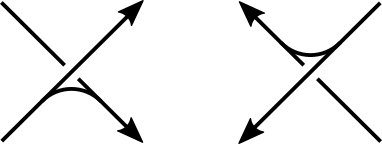}
    \caption{Inverted highway diagrams of the two example crossings}
    \label{fig:2newcrossingsHighwayDiagrams}
\end{figure}
The analogue of Proposition \ref{prop:inversionCorrespondence} for these two can be summarized as in Fig. \ref{fig:newInversionCorrespondence}. 
\begin{figure}[ht]
    \centering
    \includegraphics[scale=0.4]{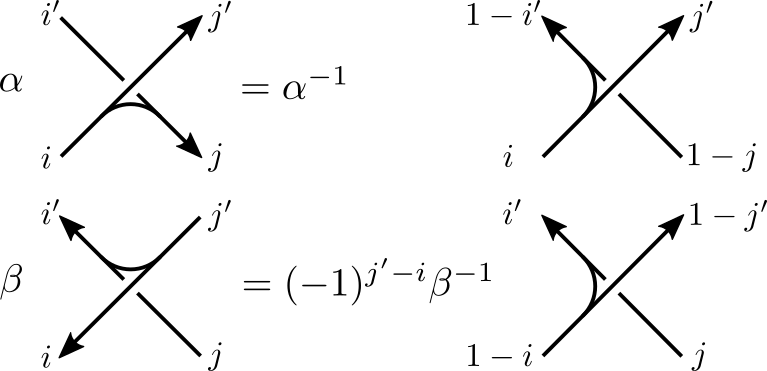}
    \caption{Analogue of Proposition \ref{prop:inversionCorrespondence} for the two example crossings}
    \label{fig:newInversionCorrespondence}
\end{figure}
Observe that while there is no sign factor for the first one, while there is a sign factor of $(-1)^{j'-i}$ for the second one. As before, the sign factors count $-1$ to the number of times a given multi-cycle is coming out of the $-$-signed segments. 
Therefore, by exactly the same argument as before, we see that $(-1)^s$ is the correct overall sign factor we get when we invert the state sum. 
\end{proof}

In Appendix \ref{sec:fiberedKnotsInversionData}, we find that Theorem \ref{thm:beyondHomogeneousBraids} can be used to compute $F_K$ for all fibered knots up to 10 crossings. 
This provides good evidence to the following conjecture. 
\begin{conj}
For any fibered knot $K$, the coefficients of $F_K$ are Laurent polynomials in $q$. 
\end{conj}
In fact, this is what is expected from the point of view of holomorphic curve counting. In \cite{EGGKPS} (see also \cite{DE}), $F_K$ was interpreted as a count of open holomorphic curves in $T^*S^3$ (or the resolved conifold) in the presence of the knot complement Lagrangian. Moreover, analogously to the knots-quivers correspondence \cite{KRSS0,KRSS,EKL,EKL2} which is for the knot conormal instead of the knot complement, $F_K$ is conjectured \cite{Kuch} to be realized as the quiver partition function of a quiver generated by a set of basic disks. In case $K$ is fibered, the knot complement Lagrangian can be shifted off completely from the zero-section $S^3$, and there is a positive area argument\footnote{I thank Tobias Ekholm for explaining this to me.} which suggests that in this case all the quiver nodes must have positive $x$-degree, meaning the coefficients of $F_K$ as a power series in $x$ must be Laurent polynomials.

\section{Inverted Habiro series}\label{sec:invertedHabiro}
\subsection{Inverting the Habiro series}
It is a well-known theorem of Habiro \cite{Hab} that for any knot $K$, there is a sequence of Laurent polynomials $a_m(K)\in \mathbb{Z}[q,q^{-1}]$ such that the colored Jones polynomials $J_K(V_{n})$, colored by the $n$-dimensional irreducible representation $V_n$ of $\mathfrak{sl}_2$, can be decomposed as
\[
J_K(V_{n}) = \sum_{m=0}^{\infty}a_m(K)\prod_{j=1}^{m}(x+x^{-1}-q^j-q^{-j})\bigg\vert_{x=q^n}.
\]

The purpose of this section is to make a curious observation that for some simple knots and links the sequence $\{a_m(K)\}_{m\geq 0}$ can be naturally extended to negative values of $m$. 
The simplest example is the figure-eight knot. In this case we have $a_m(\mathbf{4}_1)=1$ for any $m\geq 0$, so it is very natural to set $a_m(\mathbf{4}_1)=1$ for any $m\in \mathbb{Z}$. Once we have such an extension of the sequence $\{a_m(K)\}$, we can ``invert" the Habiro series in the following way:
\[
\sum_{m=0}^{\infty}a_m(K)\prod_{j=1}^{m}(x+x^{-1}-q^j-q^{-j}) \leadsto -\sum_{m=1}^{\infty}\frac{a_{-m}(K)}{\prod_{j=0}^{m-1}(x+x^{-1}-q^j-q^{-j})}.
\]
The series on the right-hand side is what we will call an \emph{inverted Habiro series}. As we will see through examples, $a_{m}(K)$ with $m<0$ are in general not necessarily Laurent polynomials, but rather Laurent power series in $q$ with integer coefficients. 

A priori the ``inverted Habiro series" is a very different object compared to the original Habiro series. For one thing, the specialization $x=q^n$ doesn't make sense any more for the inverted Habiro series, as it has poles at $q^\mathbb{Z}$. For another thing, the inverted Habiro series can be expanded into a power series in $x$ or $x^{-1}$, which is not possible for the usual Habiro series. The distinction is clearer in the classical limit $q\rightarrow 1$ where the usual Habiro series take values in $\mathbb{Z}[[x+x^{-1}-2]]$ (a completion at a finite place $x+x^{-1}=2$) whereas the inverted Habiro series take values in $\mathbb{Q}[[\frac{1}{x+x^{-1}-2}]]$ (a completion at an infinite place $x+x^{-1}=\infty$). Morally, this is the same as what we did in the previous section; whether it is a state sum or a Habiro series, we are inverting it to make it convergent near $x+x^{-1}=\infty$. 

Our main conjecture of this section is that $F_K$ can be thought of as the inverted Habiro series. 
\begin{conj}
For any knot $K$ with $\Delta_K(x)\neq 1$, it has an inverted Habiro series, and it agrees with the $F_K$ in the sense that 
\[
F_K(x,q) = -(x^{\frac{1}{2}}-x^{-\frac{1}{2}})\sum_{m=1}^{\infty}\frac{a_{-m}(K)}{\prod_{j=0}^{m-1}(x+x^{-1}-q^j-q^{-j})}
\]
when the right-hand side is expanded into a power series in $x$. 
\end{conj}
\begin{rmk}
In this conjecture, we have imposed the condition that $\deg \Delta_K(x)>0$ to make sure that $\frac{1}{\Delta_K(x)} = O(x)$ near $x=0$. 
\end{rmk}

Before we proceed to study various examples, we summarize some useful identities in the context of inverted Habiro series. 

\begin{prop}\label{prop:powerSeriesToInvertedHabiro}
If we write
\[
-(x^{\frac{1}{2}}-x^{-\frac{1}{2}})\sum_{m=1}^{\infty}\frac{a_{-m}(K;q)}{\prod_{k=0}^{m-1}(x+x^{-1}-q^k-q^{-k})} = x^{\frac{1}{2}}\sum_{j\geq 0}f_j(K;q)x^j,
\]
then
\begin{align}\label{eq:powerSeriesToInvertedHabiro}
f_j(K;q) &= \sum_{i=0}^{j}\qbin{j+i}{2i}a_{-i-1}(K;q),\\
a_{-i-1}(K;q) &= \sum_{j=0}^{i} (-1)^{i+j}\qbin{2i}{i-j}\frac{[2j+1]}{[i+j+1]} f_j(K;q).\nonumber
\end{align}
\end{prop}
Using this proposition and Theorem \ref{thm:mainTheorem}, we can easily compute the inverted Habiro coefficients for any homogeneous braid knot that is not an unknot. 
\begin{rmk}
Proposition \ref{prop:powerSeriesToInvertedHabiro} says that the inverted Habiro series and the power series in $x$ contain the same amount of information. Still, the inverted Habiro series is an especially nice way of presenting $F_K$, as it is manifestly Weyl-symmetric. What is more, it helps us uncover various surgery formulas, as we we will see in later subsections. 
\end{rmk}

\subsection{Inverted Habiro series for some simple knots and links}\label{subsec:invertedHabiroForSimpleKnots}
The $a_m(K)$ for the figure-eight and the trefoil knots of both handedness are particularly simple:
\[
a_m(\mathbf{4}_1) = 1,\quad a_m(\mathbf{3}_1^l) = (-1)^m q^{\frac{m(m+3)}{2}},\quad a_m(\mathbf{3}_1^r) = (-1)^m q^{-\frac{m(m+3)}{2}}.
\]
It is straightforward to extend these to negative $m$. From this, we obtain the inverted Habiro expansions of the corresponding $F_K$'s:
\begin{equation}
    F_{\mathbf{4}_1}(x,q) = -(x^{\frac{1}{2}}-x^{-\frac{1}{2}})\sum_{m=1}^{\infty}\frac{1}{\prod_{j=0}^{m-1}(x+x^{-1}-q^j-q^{-j})},
\end{equation}
\begin{align}
F_{\mathbf{3}_1^l}(x,q) &= -(x^{\frac{1}{2}}-x^{-\frac{1}{2}})\sum_{m=1}^{\infty}\frac{(-1)^m q^{\frac{m(m-3)}{2}}}{\prod_{j=0}^{m-1}(x+x^{-1}-q^j-q^{-j})}\\
&= q^{-1}(x^{\frac{1}{2}}-x^{-\frac{1}{2}})\sum_{n\geq 0}\frac{(-1)^n q^{\frac{n(n-1)}{2}}}{\prod_{j=0}^{n}(x+x^{-1}-q^j-q^{-j})},\nonumber
\end{align}
\begin{align}
F_{\mathbf{3}_1^r}(x,q) &= -(x^{\frac{1}{2}}-x^{-\frac{1}{2}})\sum_{m=1}^{\infty}\frac{(-1)^m q^{-\frac{m(m-3)}{2}}}{\prod_{j=0}^{m-1}(x+x^{-1}-q^j-q^{-j})}\\
&= q(x^{\frac{1}{2}}-x^{-\frac{1}{2}})\sum_{n\geq 0}\frac{(-1)^n q^{-\frac{n(n-1)}{2}}}{\prod_{j=0}^{n}(x+x^{-1}-q^j-q^{-j})}.\nonumber
\end{align}
Their power series expansions agree with the known results in \cite{GM}. 

The Habiro coefficients for the Whitehead link and the Borromean rings are given in \cite{Hab0}. Inverting them, we obtain the following expressions for the corresponding $F_K$'s:
\begin{equation}
F_{\mathbf{Wh}}(x_1,x_2,q) = q^{-\frac{1}{2}}(x_1^{\frac{1}{2}}-x_1^{-\frac{1}{2}})(x_2^{\frac{1}{2}}-x_2^{-\frac{1}{2}})\sum_{n\geq 0}\frac{(-1)^nq^{-\frac{n(n+1)}{2}}(q^{n+1})_n}{\prod_{j=0}^{n}(x_1+x_1^{-1}-q^j-q^{-j})(x_2+x_2^{-1}-q^j-q^{-j})},
\end{equation}
\begin{equation}
F_{\mathbf{Bor}}(x_1,x_2,x_3,q) =
{\scriptstyle (x_1^{\frac{1}{2}}-x_1^{-\frac{1}{2}})(x_2^{\frac{1}{2}}-x_2^{-\frac{1}{2}})(x_3^{\frac{1}{2}}-x_3^{-\frac{1}{2}})\sum_{n\geq 0}\frac{(-1)^nq^{-\frac{3n^2+n}{2}}(q^{n+1})_n^2}{\prod_{j=0}^{n}(x_1+x_1^{-1}-q^j-q^{-j})(x_2+x_2^{-1}-q^j-q^{-j})(x_3+x_3^{-1}-q^j-q^{-j})}}.
\end{equation}
They agree with the results in Examples \ref{eg:Whitehead} and \ref{eg:Borromean} (the Whitehead link formula presented here agrees with \cite{Park} and is mirror to the one we used in Example \ref{eg:Whitehead}).

\subsection{Positive integer surgeries}
One surprising aspect of inverted Habiro series is that it uncovers some surgery formulas that were not known before. 
\begin{prop}
The following identity holds
\begin{equation}\label{eq:minus1surgery}
\frac{1}{\prod_{j=1}^{n}(x+x^{-1}-q^j-q^{-j})}\bigg\vert_{x^u\mapsto q^{u^2}} = \frac{q^{n^2}}{(q^{n+1})_n},
\end{equation}
where $f(x)\bigg\vert_{x^u \mapsto q^{u^2}}$ is a shorthand notation for ``first expand $f(x)$ into a power series in $x$ and then replace each $x^u$ by $q^{u^2}$". In other words, 
\[
\sum_{k=0}^{\infty}(q^{(n+k)^2}-q^{(n+k+1)^2})\qbin{2n+k}{2n} = \frac{q^{n^2}}{(q^{n+1})_n}.
\]
\end{prop}
\begin{cor}[$-1$-surgery formula]
Suppose that 
\[
F_K(x,q) = -(x^{\frac{1}{2}}-x^{-\frac{1}{2}})\sum_{m=1}^{\infty}\frac{a_{-m}(K)}{\prod_{j=0}^{m-1}(x+x^{-1}-q^j-q^{-j})}.
\]
Then the $-1$-surgery formula of \cite{GM} can be written as
\[
\hat{Z}(S^3_{-1}(K)) = -q^{-\frac{1}{2}}\sum_{n\geq 0} a_{-n-1}(K)\frac{q^{n^2}}{(q^{n+1})_n}
\]
\end{cor}
\begin{eg}[$-1$-surgery on $\mathbf{4}_1$]
Let's take a look at the $-1$-surgery on the figure-eight knot $\mathbf{4}_1$. Since the inverted Habiro coefficients of $\mathbf{4}_1$ are all $1$, we immediately get the following formula
\[
\hat{Z}(-\Sigma(2,3,7)) = \hat{Z}(S^3_{-1}(\mathbf{4}_1)) = -q^{-\frac{1}{2}}\sum_{n\geq 0}\frac{q^{n^2}}{(q^{n+1})_n}. 
\]
Up to a prefactor, this is $\mathcal{F}_0(q)$, one of Ramanujan's mock theta functions of order 7. 
This is a result obtained in \cite{GM}, but the inverted Habiro series explains why we have such a nice expression. 
\end{eg}

It is a very useful fact that the right-hand side of \eqref{eq:minus1surgery} is a rational function in $q$. This means that even if we replace $q$ by $q^{-1}$, the right-hand side will still make sense as a power series in $q$. 
\begin{cor}
We have the identity
\begin{equation}\label{eq:plus1surgery}
\frac{1}{\prod_{j=1}^{n}(x+x^{-1}-q^j-q^{-j})}\bigg\vert_{x^u\mapsto q^{-u^2}} = \frac{(-1)^n q^{\frac{n(n+1)}{2}}}{(q^{n+1})_n}.
\end{equation}
\end{cor}
Since the transform $x^u \mapsto q^{-u^2}$ when applied to $(x^\frac{1}{2}-x^{-\frac{1}{2}})F_K(x,q)$ is the $+1$-surgery formula given in \cite{GM}, we can make use of this identity to do $+1$-surgery. The result, written in terms of inverted Habiro coefficients, looks like
\begin{equation}\label{eq:plus1surgeryformula}
\hat{Z}(S^3_{+1}(K)) = q^{\frac{1}{2}}\sum_{n\geq 0} a_{-n-1}(K)\frac{(-1)^n q^{\frac{n(n+1)}{2}}}{(q^{n+1})_n}.
\end{equation}
However, one should be noted of the important difference between the $+1$-surgery formula of \cite{GM} and the above formula. In \eqref{eq:plus1surgeryformula}, we are making an implicit regularization of the $q$-series. This is because the left-hand side of \eqref{eq:plus1surgery} is a power series in $q^{-1}$, whereas we are using the right-hand side of \eqref{eq:plus1surgery} as a series in $q$. This is more evident when we write the transform \eqref{eq:plus1surgeryformula} in terms of $F_K(x,q) = x^{\frac{1}{2}}\sum_{j\geq 0}f_j(K;q)x^j$ using Proposition \ref{prop:powerSeriesToInvertedHabiro}. The result is 
\begin{equation}\label{eq:plus1surgeryregularization}
\hat{Z}(S^3_{+1}(K)) = q^{\frac{1}{2}}\sum_{j\geq 0}f_j(K;q)(q^{-j^2}-q^{-(j+1)^2})\qty(1-\frac{\sum_{|k|\leq j}(-1)^{k}q^{\frac{k(3k+1)}{2}}}{(q)_\infty}).
\end{equation}
In this form, it is clear what the effect of the implicit regularization \eqref{eq:plus1surgery} is; it adds the `regularization factor' of $1-\frac{\sum_{|k|\leq j}(-1)^{k}q^{\frac{k(3k+1)}{2}}}{(q)_\infty}$. Note that \eqref{eq:plus1surgeryregularization} can be written in the following form
\begin{equation}
\hat{Z}(S^3_{+1}(K)) = \frac{q^{\frac{1}{2}}}{(q)_\infty}\sum_{\substack{j\geq 0 \\ |k| > j}}f_j(K;q)(q^{-j^2}-q^{-(j+1)^2})(-1)^{k}q^{\frac{k(3k+1)}{2}}.
\end{equation}
This expression hints that the $+1$-surgery formula is related to indefinite theta functions. We study this connection further in Section \ref{sec:indefiniteTheta}. 

Since \eqref{eq:plus1surgeryregularization} is different from the $+1$-surgery formula of \cite{GM}, it may seem like we have a contradiction. However, recall that \cite{GM} conjectures that their surgery formula holds when it converges. Of course, when it converges, there is no need to make regularization. So we make the following conjecture. 
\begin{conj}[Regularized $+1$-surgery formula]
When the $+1$-surgery formula of \cite{GM} converges, we need to use their formula (i.e. without regularization factor). When it does not converge, we can to make regularization according to \eqref{eq:plus1surgeryregularization}, as long as the regularization converges. 
\end{conj}
We have checked this conjecture through various examples. 
\begin{eg}[$+1$-surgery on $\mathbf{4}_1$]
Naive application of the $+1$-surgery formula of \cite{GM} to $F_{\mathbf{4}_1}$ does not converge. So we need to use our regularization. From \eqref{eq:plus1surgeryformula}, we immediately get
\[
\hat{Z}(\Sigma(2,3,7)) = \hat{Z}(S^3_{+1}(\mathbf{4}_1)) = q^{\frac{1}{2}}\sum_{n\geq 0}\frac{(-1)^nq^{\frac{n(n+1)}{2}}}{(q^{n+1})_n},
\]
which agrees with the computation from the plumbing description of $\Sigma(2,3,7)$ in \cite{GM}. 
\end{eg}

\begin{eg}[$+1$-surgery on Whitehead link]
Consider the Whitehead link which is obtained by $-1$-surgery on a component of the Borromean rings (this is the same one considered in \cite{Park} and is mirror to Example \ref{eg:Whitehead}). If we apply the $+1$-surgery of \cite{GM}, then it stabilizes to $F_{\mathbf{4}_1}$ but does not converge to it. This is because even if we add more and more terms, there are always terms with higher and higher power of $q^{-1}$. 
This issue is solved by using the regularized $+1$-surgery formula; in this way, the result actually converges to $F_{\mathbf{4}_1}$. 
\end{eg}

\begin{eg}[$+1$-surgery on $\mathbf{L7a1}$]
As briefly mentioned in Example \ref{eg:L7a1}, the $+1$-surgery on the $x_2$-component of $\mathbf{L7a1}$ is the $\mathbf{6}_2$ knot. Indeed, applying the regularized $+1$-surgery formula, the result converges to $F_{\mathbf{6}_2}(x_1,q)$. 
\end{eg}

Now we study some other integer surgeries in a similar fashion. 

\begin{prop}
The $q$-series
\[
\frac{1}{\prod_{j=1}^{n}(x+x^{-1}-q^j-q^{-j})}\bigg\vert_{x^u\mapsto \delta_{b,u(\text{mod }p)}q^{\frac{u^2}{p}}}
\]
is a rational function function in $q$. Here, $f(x)\bigg\vert_{x^u\mapsto \delta_{b,u(\text{mod }p)}q^{\frac{u^2}{p}}}$ is a shorthand notation for ``first expand $f(x)$ into a power series in $x$ and then replace each $x^u$ by $q^{\frac{u^2}{p}}$ whenever $b = u\,(\text{mod }p)$ and by $0$ otherwise''.
\end{prop}
\begin{proof}
We start by putting it in a form that is easier to deal with. 
\begin{align*}
&(q)_{2n}\frac{1}{\prod_{j=1}^{n}(x+x^{-1}-q^j-q^{-j})}\bigg\vert_{x^u\mapsto \delta_{b,u(\text{mod }p)}q^{\frac{u^2}{p}}}\\
&= -(q)_{2n}(x^{\frac{1}{2}}-x^{-\frac{1}{2}})\sum_{k\geq 0}x^{n+k+\frac{1}{2}}\qbin{2n+k}{2n}\bigg\vert_{x^u\mapsto \delta_{b,u(\text{mod }p)}q^{\frac{u^2}{p}}}\\
&= \sum_{k\geqq 0}(x^{n+k}-x^{n+k+1})q^{-nk}(1-q^{k+1})\cdots(1-q^{k+2n})\bigg\vert_{x^u\mapsto \delta_{b,u(\text{mod }p)}q^{\frac{u^2}{p}}}\\
&= \sum_{\substack{k\geq 0 \\ k\equiv b-n\text{ mod }p}}q^{\frac{(n+k)^2}{p}}q^{-nk}(1-q^{k+1})\cdots(1-q^{k+2n})\\
&\quad - \sum_{\substack{k\geq 0 \\ k\equiv b-n-1\text{ mod }p}}q^{\frac{(n+k+1)^2}{p}}q^{-nk}(1-q^{k+1})\cdots(1-q^{k+2n})
\end{align*}
Expanding $(1-q^{k+1})\cdots(1-q^{k+2n})$ on both sides, there are $2^{2n}$ terms on each side of the summation. There is a natural pairing between the terms on the left side with the terms on the right side, given by the following rule. Any term on the left side can be expressed as a sequence of signs, where the $i$-th sign would represent whether we choose $1$ or $-q^{k+i}$ from $(1-q^{k+i})$. Then we flip all the signs and reverse the order. This new sequence will correspond to a term on the right summation. If we started with a term $q^{\frac{(n+k)^2}{p}+Ak+B}$ from the left summation, then the corresponding term on the right summation is $-q^{\frac{(n+k+1)^2}{p}-Ak+(B-(2n+1)A)}$. Notice that
\begin{align*}
q^{\frac{(n+k)^2}{p}+Ak+B} &= q^{\frac{1}{p}(n+k)(n+k+pA)+(B-An)},\\
-q^{\frac{(n+k+1)^2}{p}-Ak+(B-(2n+1)A)} &= -q^{\frac{1}{p}(n+k+1-pA)(n+k+1)+(B-An)}.
\end{align*}
This means that the summation of these pairs over $k$ telescopes, leaving only finitely many terms. Therefore the expression we started with is a Laurent polynomial in $q$. 
\end{proof}

This proposition itself is enough to study positive and negative integer surgeries, but we should remark that, based on experiments, it seems we can say more about the structure of the rational function; it takes the following form:
\begin{equation}
\frac{1}{\prod_{j=1}^{n}(x+x^{-1}-q^j-q^{-j})}\bigg\vert_{x^u\mapsto \delta_{b,u(\text{mod }p)}q^{\frac{u^2}{p}}} = \frac{q^{n^2}}{(q^{n+1})_n}q^{-\frac{b(p-b)}{p}}P_{n}^{p,b},
\end{equation}
where $P_{n}^{p,b}\in \mathbb{Z}[q^{-1}]$ is a polynomial with non-negative coefficients of degree at most $p\lfloor \frac{n^2}{4}\rfloor$ whose classical limit is $p^{n-1}$ for any $n\geq 1$. 
In Table \ref{tab:Ppbnexamples}, we list the first few polynomials $P_{n}^{p,b}$. 
\begin{table}[ht]
    \centering
    \begin{tabular}{c | c}
        $n$ &  \\
        \hline\hline
        $0$ & $P_0^{p,b}=\delta_{b,0}$\\
        \hline
        $1$ & $P_1^{p,b}=1$\\
        \hline
        $2$ & \begin{tabular}{@{}c@{}} $P_2^{1,0}=1$ \\ \hline $P_2^{2,0}=1+q^{-2}$ \\ $P_2^{2,1}=1+q^{-1}$ \\ \hline $P_2^{3,0} = 1+q^{-2}+q^{-3}$\\ $P_2^{3,1} = 1+q^{-1}+q^{-2}$ \end{tabular}\\
        \hline
        $3$ & \begin{tabular}{@{}c@{}} $P_3^{1,0}=1$ \\ \hline $P_3^{2,0}=1+q^{-2}+q^{-3}+q^{-4}$ \\ $P_3^{2,1}=1+q^{-1}+q^{-2}+q^{-4}$ \\ \hline $P_3^{3,0} = 1+q^{-2}+2q^{-3}+2q^{-4}+q^{-5}+2q^{-6}$\\ $P_3^{3,1} = 1+q^{-1}+2q^{-2}+q^{-3}+2q^{-4}+q^{-5}+q^{-6}$ \end{tabular}
    \end{tabular}
    \caption{Some examples of $P_n^{p,b}$}
    \label{tab:Ppbnexamples}
\end{table}

Replacing $q$ with $q^{-1}$, the right-hand side still makes sense as a power series in $q$. 
\begin{cor} 
We have the identity
\begin{equation}\label{eq:psurgery}
\frac{1}{\prod_{j=1}^{n}(x+x^{-1}-q^j-q^{-j})}\bigg\vert_{x^u\mapsto \delta_{b,u(\text{mod }p)}q^{-\frac{u^2}{p}}} = \frac{(-1)^n q^{\frac{n(n+1)}{2}}}{(q^{n+1})_n}q^{\frac{b(p-b)}{p}}P_{n}^{p,b}(q^{-1}).
\end{equation}
\end{cor}
Just like we made use of \eqref{eq:plus1surgery} to conjecture a regularized version of the $+1$-surgery formula, we conjecture that we can use \eqref{eq:psurgery} to regularize $p$-surgery. 

\begin{conj}[Regularized $+p$-surgery formula]
When the $+p$-surgery formula of \cite{GM} converges, we need to use their formula. When it does not converge, we can use \eqref{eq:psurgery} to regularize it, as long as the regularization converges. 
\end{conj}

\begin{eg}[$+1, +2, +3$-surgery on $\mathbf{4}_1$]
The $+1, +2, +3$-surgeries on $\mathbf{4}_1$ are all Seifert manifolds. They are nicely summarized in Table 9 of \cite{GM}. In that table, $\hat{Z}$ for those $3$-manifolds were computed from their plumbing descriptions. The surgery formula of \cite{GM} cannot be applied directly to compute them as $+1, +2, +3$-surgeries on $\mathbf{4}_1$, as it gives non-convergent results. Instead, we can use the regularized $+p$-surgery formula, and we find that the result agrees perfectly with the computation from plumbing descriptions. 
\end{eg}

\subsection{Quantum $C$-polynomial recursion and positive small surgeries}
So far, we have discussed how to use inverted Habiro series to study various integer surgeries, but we haven't discussed much how to compute the inverted Habiro series. Of course, once we know $F_K$, we can use Proposition \ref{prop:powerSeriesToInvertedHabiro} to compute inverted Habiro coefficients, but is there a way to compute the inverted Habiro coefficients just from the Habiro coefficients? 
Sometimes this is possible by solving the quantum (non-commutative) $C$-polynomial recursion, which is the topic of this subsection. 

It was first shown in \cite{GL} that the Habiro coefficients $\{a_m(K)\}_{m\geq 0}$ are $q$-holonomic, and the recurrence relation was further studied in \cite{GS}, and it was named a $C$-polynomial. A quantum $C$-polynomial $\hat{C}_K(\hat{E},\hat{Q},q)$ is written in terms of the $q$-commutative operators $\hat{Q}$ and $\hat{E}$. These operators act on the set of discrete functions by
\[
(\hat{Q}f)(m) = q^mf(m),\quad (\hat{E}f)(m) = f(m+1),
\]
and they satisfy the following $q$-commutation relation
\[
\hat{E}\hat{Q} = q\hat{Q}\hat{E}.
\]
Let $E$ be a variable where $\hat{Q}$ and $\hat{E}$ act by
\[
\hat{Q}E^{-m} = q^m E^{-m},\quad \hat{E}E^{-m} = E^{-(m-1)}.
\]
It is useful to combine the Habiro coefficients $\{a_m(K)\}_{m\geq 0}$ into a series
\[
\sum_{m \geq 0}a_m(K)E^{-m}.
\]
Since $\hat{C}_K(\hat{E},\hat{Q},q)$ defines a recurrence relation for $\{a_m(K)\}_{m\geq 0}$, when we apply $\hat{C}_K$ to the above series, all but finitely many terms will cancel out. However, in general it doesn't vanish completely, because the boundary terms survive. 
The non-vanishing of $\hat{C}_K(\hat{E},\hat{Q},q)\sum_{m \geq 0}a_m(K)E^{-m}$ is actually very useful for our purpose, because we can take it as the starting point of the recursion and use it to extend the sequence $\{a_m(K)\}_{m\geq 0}$ to negative $m$. More precisely, we want to extend it into a bilateral sequence $\{a_m(K)\}_{m\in \mathbb{Z}}$ in such a way that 
\begin{equation}
\hat{C}_K(\hat{E},\hat{Q},q)\sum_{m\in \mathbb{Z}}a_m(K)E^{-m} = 0.
\end{equation}
Since explicit expressions for the quantum $C$-polynomials for twist knots are given in \cite{GS}, we will use them and demonstrate how this strategy works. 

The simplest cases are the trefoil knot and the figure-eight knot. In those cases, we find that solving the quantum $C$-polynomial recursion, there is a unique way to extend the series $\sum_{m\geq 0}a_m(K)E^{-m}$ into a bilateral series $\sum_{m\in \mathbb{Z}}a_m(K)E^{-m}$, and the result agrees with what we found in Section \ref{subsec:invertedHabiroForSimpleKnots}. This is analogous to how in \cite{EGGKPS} we only needed the very first term to determine the full power series $F_K(x,q)$ for the trefoil and the figure-eight knot using quantum $A$-polynomial recursion. 

Other twist knots are more interesting. Let's focus on the two-twist knot $K_2 = \mathbf{5}_2$ for the moment. In this case, the series $\sum_{m\geq 0}a_m(\mathbf{5}_2)E^{-m}$ looks like
\[
1 + (-q^2-q^4)E^{-1} + (q^5+q^7+q^8+q^{11})E^{-2} +(-q^9-q^{11}-q^{12}-q^{13}-q^{15}-q^{16}-q^{17}-q^{21})E^{-3} + \cdots,
\]
and the quantum $C$-polynomial is given by
\[
\hat{C}_{\mathbf{5}_2}(\hat{E},\hat{Q},q) = \hat{E}^2 + (q^2+q^3)\hat{E}\hat{Q} + (q^6-q^3\hat{E})\hat{Q}^2 + (-q^7+q^4\hat{E})\hat{Q}^3.
\]
When we try to solve the recursion, we quickly realize that $a_{-1}(\mathbf{5}_2)$ cannot be determined by $\{a_m(\mathbf{5}_2)\}_{m\geq 0}$, but once we know $a_{-1}(\mathbf{5}_2)$ all the other $a_m(\mathbf{5}_2)$ with $m < -1$ are uniquely determined. This is analogous to how we need to know the first two terms to determine the full power series $F_{\mathbf{5}_2}(x,q)$ if we were to solve it using quantum $A$-polynomial recursion. 

So, how do we determine $a_{-1}(\mathbf{5}_2)$? It turns out this can be done by imposing a certain boundary condition on the bilateral sequence. More precisely, we take the ansatz where we express $a_{-1}(K)$ in terms of $\{a_m(K)\}_{m\leq -M}$ and take the limit where $M$ goes to $\infty$. 
Explicitly, we have the following sequence of relations coming from the quantum $C$-polynomial recursion:
\begin{align*}
a_{-1}(\mathbf{5}_2) &= \frac{1}{1+q}\qty(-q^{-1}+(1-q)a_{-2}(\mathbf{5}_2))\\
a_{-2}(\mathbf{5}_2) &= \frac{1}{1+q^2+q^3+q^4}\qty(q+(1+q-q^2-q^3)a_{-3}(\mathbf{5}_2))\\
a_{-3}(\mathbf{5}_2) &= \frac{1}{1+q^3+q^4+q^5+q^6+q^7+q^8+q^9}\qty(-q^6+(1+q^2+q^4-q^5-q^6-q^7)a_{-4}(\mathbf{5}_2)),
\end{align*}
and so on. Starting with the first equation, we can plug in the second equation to express $a_{-1}(\mathbf{5}_2)$ in terms of $a_{-3}(\mathbf{5}_2)$. Plugging in the third equation to that would give an expression of $a_{-1}(\mathbf{5}_2)$ in terms of $a_{-4}(\mathbf{5}_2)$, and we can proceed in the same way indefinitely. Taking the limit, we have the following ansatz for $a_{-1}(\mathbf{5}_2)$:
\begin{equation}\label{eq:aminus1of52}
a_{-1}(\mathbf{5}_2)=
\frac{-q^{-1}}{1+q}\qty(1+\frac{-(1-q)q^2}{1+q^2+q^3+q^4}\qty(1+\frac{-(1+q-q^2-q^3)q^5}{1+q^3+q^4+q^5+q^6+q^7+q^8+q^9}\qty(\cdots))).
\end{equation}
Expanding this into a power series in $q$, we find
\[
a_{-1}(\mathbf{5}_2) = -q^{-1}+1-q^2+q^5-q^9+q^{14}-q^{20}+q^{27}-O(q^{35}),
\]
which agrees perfectly with $F_{\mathbf{5}_2}(x,q)$ obtained in \cite{Park} (our $\mathbf{5}_2$ is the knot $m(\mathbf{5}_2)$ in \cite{Park}). 

The same type of ansatz seems to work for all twist knots. In general, for the $n$-twist knot $K_n$, $a_{l}(K_n)$ can be expressed in terms of $\{a_{l-k}(K_n)\}_{1\leq k\leq |n|-1}$. Using these relations, we can always express $a_{l}(K_n)$ in terms of $\{a_m(K_n)\}_{m\leq -M}$ for any $M$. Sending $M$ to $\infty$, we get the ansatz. 

All these ansatz can be obtained as a limit of a sequence of rational functions in $q$. As we emphasized in previous subsections, having a rational function is especially useful as it allows us to study the mirror knot and the orientation-reversed $3$-manifold. For instance, we can simply replace $q$ by $q^{-1}$ in \eqref{eq:aminus1of52} and still expand it as a power series in $q$. As a result, we get a candidate for $a_{-1}(m(\mathbf{5}_2))$:
\[
a_{-1}(m(\mathbf{5}_2)) = -q^2+q^4-q^5+2q^6-2q^7+q^8-3q^{10}+O(q^{11}).
\]
By construction, the asymptotic expansion of this series near each root of unity agrees with that of $a_{-1}(\mathbf{5}_2)\big\vert_{q\rightarrow q^{-1}}$. 
What we are doing here is morally very similar to what we did when we described the regularized versions of positive integer surgery formulas. We are regularizing $a_{-1}(\mathbf{5}_2)\big\vert_{q\rightarrow q^{-1}}$ into a power series in $q$ by using the fact that it can be naturally expressed as the limit of a sequence of rational functions. 

In a similar fashion, we get candidates of the inverted Habiro coefficients for all mirror twist knots. For instance, the mirror of 
\begin{align*}
a_{-1}(\mathbf{7}_2) &= -q^{-1}+1-q^4+q^7-q^{15}+q^{20}-q^{32}+q^{39}-O(q^{55}),\\
a_{-2}(\mathbf{7}_2) &= q^3-q^5-q^6+q^8+q^{13}+q^{14}-q^{16}-q^{17}-O(q^{18})
\end{align*}
are
\begin{align*}
a_{-1}(m(\mathbf{7}_2)) &= -q^3+q^5-q^7+2q^8-2q^{10}+2q^{12}+O(q^{13}),\\
a_{-2}(m(\mathbf{7}_2)) &= q^5-q^8-q^9-q^{11}+q^{13}+O(q^{14}).
\end{align*}

Since all these twist knots can be obtained from the Whitehead link by doing $\frac{1}{r}$-surgery for various $r$, we can use these results to reverse-engineer the regularized $\frac{1}{r}$-surgery formula. It turns out, the end result is very simple. 
\begin{conj}[Regularized $+\frac{1}{r}$-surgery formula]
When the $+\frac{1}{r}$-surgery formula of \cite{GM} converges, we need to use their formula. When it does not converge, we can regularize it in the following way, as long as this regularization converges:
\begin{equation}
\hat{Z}(S^3_{+1/r}(K)) = q^{\frac{r+r^{-1}}{4}}\sum_{j\geq 0}f_j(K;q)(q^{-r(j+\frac{1}{2}-\frac{1}{2r})^2}-q^{-r(j+\frac{1}{2}+\frac{1}{2r})^2})\qty(1 - \frac{\sum_{|k|\leq j}(-1)^kq^{\frac{k((2r+1)k+1)}{2}}}{f(-q^r,-q^{r+1})})
\end{equation}
where $F_K(x,q) = x^{\frac{1}{2}}\sum_{j\geq 0}f_j(K;q)x^j$, and 
\[
f(a,b) := \sum_{n\in \mathbb{Z}}a^{\frac{n(n+1)}{2}}b^{\frac{n(n-1)}{2}} = (-a;ab)_\infty (-b,ab)_\infty (ab;ab)_\infty
\]
denotes the Ramanujan theta function. 
\end{conj}
As a small consistency check, we find that the three surgery descriptions $S^3_{\frac{1}{2}}(\mathbf{4}_1) = S^3_{-1}(m(K_2)) = S^3_{+1}(K_{-2})$ of the same manifold give the same result for $\hat{Z}$, which is
\[q^{3/2}(1-2q^2+q^3-3q^4+4q^5-q^6+q^7+5q^8+O(q^9)).\]

\begin{rmk}
Using the positive surgery formulas we have discussed so far, we can generate a variety of ``false-mock'' pairs of $\hat{Z}$'s in the sense of \cite{CCFGH}. 
It would be interesting to study their modular properties in detail. 
\end{rmk}

\subsection{Inverted Habiro series in higher rank}
When we restrict our attention to symmetric representations, the higher rank analogue of the cyclotomic expansion is known. See \cite{IMMM,MMM} and references therein. The $q$-holonomicity of these cyclotomic coefficients follows from the $q$-holonomicity of the colored HOMFLY-PT polynomials. The corresponding higher rank analogue of the quantum $C$-polynomial was studied in \cite{MM}. Thanks to these known results, it is straightforward to extend our analysis in this section to the higher rank case, by solving the higher rank quantum $C$-polynomial recursions. 
For example, in case of the figure-eight knot, the reduced version of $F_{\mathbf{4}_1}^{SU(N),sym}(x,q)$ is given by
\[
F_{\mathbf{4}_1}^{SU(N),sym,red}(x,q) = -\sum_{n\geq 0}\frac{\frac{[-n][-n+1]\cdots[-n+N-3]}{[N-2]!}}{\prod_{j=0}^{n}(x^{\frac{1}{2}}q^{\frac{j}{2}}-x^{-\frac{1}{2}}q^{-\frac{j}{2}})(x^{\frac{1}{2}}q^{\frac{(N-2)-j}{2}}-x^{-\frac{1}{2}}q^{\frac{j-(N-2)}{2}})}
\]
Note that in this form the Weyl symmetry $F_K^{SU(N)}(x,q) = F_K^{SU(N)}(q^{2-N}x^{-1},q)$ of \cite{Park0} is manifest.

\section{Connection to indefinite theta functions}\label{sec:indefiniteTheta}
\subsection{$\hat{Z}$ and indefinite theta functions}
While the definition of $\hat{Z}$ for (weakly) negative definite plumbed $3$-manifolds is given in \cite{GPPV}, it has been a challenge to extend their definition to indefinite plumbings. This is mainly because the formula of \cite{GPPV}, when naively applied to a plumbed $3$-manifold that is not (weakly) negative definite, gives a non-convergent result. One approach to overcome this issue suggested by \cite{CFS} is to use indefinite theta functions with a regularization factor. In their paper, they gave one example, which is $-\Sigma(2,3,7)$. Starting from the formula of \cite{GPPV} naively applied to a plumbing description of $-\Sigma(2,3,7)$, they multiplied both the numerator and the denominator by $(q)_\infty$. The factor $(q)_\infty = \sum_{k\in \mathbb{Z}}(-1)^k q^{\frac{k(3k+1)}{2}}$ in the numerator is then interpreted as a part of an indefinite theta function. Then they inserted a regularization factor (denoted by $\rho(\mathbf{v})$ in their paper), and the resulting formula somewhat magically gives the correct $\hat{Z}(-\Sigma(2,3,7))$. 

In this section, we explore this connection to indefinite theta functions based on our observation from the previous section that the regularization factors appearing in the regularized positive surgery formulas look very much like the regularization factors we can use for indefinite theta functions. We will work out a number of explicit examples, hoping that they will eventually shed light on figuring out a formula for $\hat{Z}$ for general plumbed $3$-manifolds.

\begin{eg}[$-\Sigma(2,3,7)$]\label{eg:minusSigma237}
The $3$-manifold $-\Sigma(2,3,7)$ has various descriptions, for instance, 
\[
-\Sigma(2,3,7) = M(1;-\frac{1}{2},-\frac{1}{3},-\frac{1}{7}) = S^3_{+1}(\mathbf{3}_1^l) = S^3_{-1}(\mathbf{4}_1).
\]
For our purpose, the description as the $+1$-surgery on the left-handed trefoil knot is the most useful one. 
The naive plumbing formula \cite{GPPV} applied to $-\Sigma(2,3,7)$ gives the following expression
\begin{align*}
\hat{Z}(-\Sigma(2,3,7)) ``&\cong" \oint \frac{d x_1}{2\pi i x_1}\frac{d x_2}{2\pi i x_2}\frac{d x_3}{2\pi i x_3}\frac{d x_7}{2\pi i x_7} \frac{(x_2^{\frac{1}{2}}-x_2^{-\frac{1}{2}})(x_3^{\frac{1}{2}}-x_3^{-\frac{1}{2}})(x_7^{\frac{1}{2}}-x_7^{-\frac{1}{2}})}{x_1^{\frac{1}{2}}-x_1^{-\frac{1}{2}}}\\
&\quad\quad\times \sum_{\vec{\ell}\in M\mathbb{Z}^4 + \frac{1}{2}\vec{1}}q^{-(\vec{\ell}, M^{-1}\vec{\ell})}\vec{x}^{\vec{\ell}},
\end{align*}
where
\[
M = 
\begin{pmatrix}
1 & 1 & 1 & 1 \\
1 & 2 & 0 & 0 \\
1 & 0 & 3 & 0 \\
1 & 0 & 0 & 7
\end{pmatrix}
, \quad
M^{-1} = \begin{pmatrix}
42 & -21 & -14 & -6 \\
-21 & 11 & 7 & 3 \\
-14 & 7 & 5 & 2 \\
-6 & 3 & 2 & 1
\end{pmatrix}.
\]
The symbol $\cong$ denotes equality up to an overall sign and a power of $q$, and we wrote it in quote, because the right-hand side does not converge. 
Integrating out $x_2,x_3,x_7$, we get
\begin{equation}\label{eq:naiveMinusSigma237}
\hat{Z}(-\Sigma(2,3,7)) ``\cong"  \oint \frac{dx_1}{2\pi i x_1}\frac{1}{x_1^{\frac{1}{2}}-x_1^{-\frac{1}{2}}}\sum_{\epsilon_2,\epsilon_3,\epsilon_7=\pm 1}\sum_{\ell_1\in \mathbb{Z}+\frac{1}{2}}\epsilon_2\epsilon_3\epsilon_7\; q^{-42(\ell_1-\frac{1}{2}(\frac{1}{2}\epsilon_2+\frac{1}{3}\epsilon_3+\frac{1}{7}\epsilon_7))^2+\frac{1}{168}}x_1^{\ell_1}.
\end{equation}
Now recall that from the plumbing description of the left-handed trefoil, we have
\[
F_{\mathbf{3}_1^l}(x,q) \cong \oint \frac{d x_1}{2\pi i x_1}\frac{d x_2}{2\pi i x_2}\frac{d x_3}{2\pi i x_3}\frac{(x_2^{\frac{1}{2}}-x_2^{-\frac{1}{2}})(x_3^{\frac{1}{2}}-x_3^{-\frac{1}{2}})}{x_1^{\frac{1}{2}}-x_1^{-\frac{1}{2}}}\sum_{\vec{n}\in \mathbb{Z}^3\times\{0\}}q^{-(\vec{n}, M'\vec{n})-(\vec{n},\vec{1})}\vec{x}^{M'\vec{n}+\frac{1}{2}\vec{1}},
\]
where
\[
M' = 
\begin{pmatrix}
1 & 1 & 1 & 1 \\
1 & 2 & 0 & 0 \\
1 & 0 & 3 & 0 \\
1 & 0 & 0 & 6
\end{pmatrix}.
\]
So the power of $x=x_6$ is $n_1+\frac{1}{2}$. 
We have seen previously that the regularized $+1$ surgery formula can be written as
\[
x^{m+\frac{1}{2}} \mapsto (q^{-m^2}-q^{-(m+1)^2})\frac{\sum_{\substack{k\in \mathbb{Z}\\ |k|> |m+\frac{1}{2}|}}(-1)^kq^{\frac{k(3k-1)}{2}}}{(q)_\infty}.
\]
Therefore,
\begin{align*}
\hat{Z}(-\Sigma(2,3,7)) &\cong \oint \frac{d x_1}{2\pi i x_1}\frac{d x_2}{2\pi i x_2}\frac{d x_3}{2\pi i x_3}\frac{(x_2^{\frac{1}{2}}-x_2^{-\frac{1}{2}})(x_3^{\frac{1}{2}}-x_3^{-\frac{1}{2}})}{x_1^{\frac{1}{2}}-x_1^{-\frac{1}{2}}}\\
&\quad \times\sum_{\vec{n}\in \mathbb{Z}^3\times\{0\}}q^{-(\vec{n}, M'\vec{n})-(\vec{n},\vec{1})}x_1^{n_1+n_2+n_3+\frac{1}{2}}x_2^{n_1+2n_2+\frac{1}{2}}x_3^{n_1+3n_3+\frac{1}{2}}\\
&\quad\times (q^{-n_1^2}-q^{-(n_1+1)^2})\frac{\sum_{\substack{k\in \mathbb{Z}\\ |k|> |n_1+\frac{1}{2}|}}(-1)^kq^{\frac{k(3k-1)}{2}}}{(q)_\infty}\\
&\cong \oint \frac{d x_1}{2\pi i x_1}\frac{d x_2}{2\pi i x_2}\frac{d x_3}{2\pi i x_3}\frac{d x_7}{2\pi i x_7}\frac{(x_2^{\frac{1}{2}}-x_2^{-\frac{1}{2}})(x_3^{\frac{1}{2}}-x_3^{-\frac{1}{2}})(x_7^{\frac{1}{2}}-x_7^{-\frac{1}{2}})}{x_1^{\frac{1}{2}}-x_1^{-\frac{1}{2}}}\\
&\quad \times\sum_{\vec{n}\in \mathbb{Z}^3\times \mathbb{Z}}q^{-(\vec{n}, M'\vec{n})-n_7^2-(\vec{n},\vec{1})}x_1^{n_1+n_2+n_3+\frac{1}{2}}x_2^{n_1+2n_2+\frac{1}{2}}x_3^{n_1+3n_3+\frac{1}{2}}x_7^{n_1+n_7+\frac{1}{2}}\\
&\quad\times \frac{\sum_{\substack{k\in \mathbb{Z}\\ |k|> |n_1+\frac{1}{2}|}}(-1)^kq^{\frac{k(3k-1)}{2}}}{(q)_\infty}\\
&= \oint \frac{d x_1}{2\pi i x_1}\frac{d x_2}{2\pi i x_2}\frac{d x_3}{2\pi i x_3}\frac{d x_7}{2\pi i x_7}\frac{(x_2^{\frac{1}{2}}-x_2^{-\frac{1}{2}})(x_3^{\frac{1}{2}}-x_3^{-\frac{1}{2}})(x_7^{\frac{1}{2}}-x_7^{-\frac{1}{2}})}{x_1^{\frac{1}{2}}-x_1^{-\frac{1}{2}}}\\
&\quad \times\sum_{\vec{n}\in \mathbb{Z}^3\times \mathbb{Z}}q^{-((n_1+6n_7)^2+2(n_2-3n_7)^2+3(n_3-2n_7)^2+2(n_1+6n_7)(n_2-3n_7)+2(n_1+6n_7)(n_3-2n_7)+n_7^2)}\\
&\quad\times q^{-((n_1+6n_7)+(n_2-3n_7)+(n_3-2n_7))}\\
&\quad \times x_1^{(n_1+6n_7)+(n_2-3n_7)+(n_3-2n_7)+\frac{1}{2}}x_2^{(n_1+6n_7)+2(n_2-3n_7)+\frac{1}{2}}\\
&\quad\times x_3^{(n_1+6n_7)+3(n_3-2n_7)+\frac{1}{2}}x_7^{(n_1+6n_7)+n_7+\frac{1}{2}}\\
&\quad\times \frac{\sum_{\substack{k\in \mathbb{Z}\\ |k|> |n_1+6n_7+\frac{1}{2}|}}(-1)^kq^{\frac{k(3k-1)}{2}}}{(q)_\infty}\\
&= \oint \frac{d x_1}{2\pi i x_1}\frac{d x_2}{2\pi i x_2}\frac{d x_3}{2\pi i x_3}\frac{d x_7}{2\pi i x_7}\frac{(x_2^{\frac{1}{2}}-x_2^{-\frac{1}{2}})(x_3^{\frac{1}{2}}-x_3^{-\frac{1}{2}})(x_7^{\frac{1}{2}}-x_7^{-\frac{1}{2}})}{x_1^{\frac{1}{2}}-x_1^{-\frac{1}{2}}}\\
&\quad \times\sum_{\vec{n}\in \mathbb{Z}^3\times \mathbb{Z}}q^{-(\vec{n}, M\vec{n})-(\vec{n},\vec{1})}x^{M\vec{n}+\frac{1}{2}\vec{1}}
\frac{\sum_{\substack{k\in \mathbb{Z}\\ |k|> |n_1+6n_7+\frac{1}{2}|}}(-1)^kq^{\frac{k(3k-1)}{2}}}{(q)_\infty}.
\end{align*}
The last expression can be rewritten as
\begin{align*}
&\oint \frac{d x_1}{2\pi i x_1}\frac{d x_2}{2\pi i x_2}\frac{d x_3}{2\pi i x_3}\frac{d x_7}{2\pi i x_7}\frac{(x_2^{\frac{1}{2}}-x_2^{-\frac{1}{2}})(x_3^{\frac{1}{2}}-x_3^{-\frac{1}{2}})(x_7^{\frac{1}{2}}-x_7^{-\frac{1}{2}})}{x_1^{\frac{1}{2}}-x_1^{-\frac{1}{2}}}\\
&\quad \times \sum_{\vec{\ell}\in M\mathbb{Z}^4+\frac{1}{2}\vec{1}}q^{-(\vec{\ell},M^{-1}\vec{\ell})}\vec{x}^{\vec{\ell}}\frac{\sum_{\substack{k\in \mathbb{Z}\\ |k|>| ((6,-3,-2,0),\vec{\ell})|}}(-1)^kq^{\frac{k(3k-1)}{2}}}{(q)_\infty}\\
&\cong \oint \frac{dx_1}{2\pi i x_1}\frac{1}{x_1^{\frac{1}{2}}-x_1^{-\frac{1}{2}}}\sum_{\epsilon_2,\epsilon_3,\epsilon_7=\pm 1}\sum_{\ell_1\in \mathbb{Z}+\frac{1}{2}}\epsilon_2\epsilon_3\epsilon_7\; q^{-42\ell_1^2+(21\epsilon_2+14\epsilon_3+6\epsilon_7)\ell_1-\frac{1}{2}(7\epsilon_2\epsilon_3+3\epsilon_2\epsilon_7+2\epsilon_3\epsilon_7)}x_1^{\ell_1}\\
&\quad\times \frac{\sum_{\substack{k\in \mathbb{Z}\\ |k|> 6|\ell_1-\frac{1}{2}(\frac{1}{2}\epsilon_2 +\frac{1}{3}\epsilon_3)|}}(-1)^kq^{\frac{k(3k-1)}{2}}}{(q)_\infty}\\
&= \oint \frac{dx_1}{2\pi i x_1}\frac{1}{x_1^{\frac{1}{2}}-x_1^{-\frac{1}{2}}}\sum_{\epsilon_2,\epsilon_3,\epsilon_7=\pm 1}\sum_{\ell_1\in \mathbb{Z}+\frac{1}{2}}\epsilon_2\epsilon_3\epsilon_7\;\\ &\quad\times q^{-42(\ell_1-\frac{1}{2}(\frac{1}{2}\epsilon_2+\frac{1}{3}\epsilon_3+\frac{1}{7}\epsilon_7))^2-\frac{1}{2}(7\epsilon_2\epsilon_3+3\epsilon_2\epsilon_7+2\epsilon_3\epsilon_7)+\frac{21}{2}(\frac{1}{2}\epsilon_2+\frac{1}{3}\epsilon_3+\frac{1}{7}\epsilon_7))^2}x_1^{\ell_1}\\
&\quad\times \frac{\sum_{\substack{k\in \mathbb{Z}\\ |k|> 6|\ell_1-\frac{1}{2}(\frac{1}{2}\epsilon_2 +\frac{1}{3}\epsilon_3)|}}(-1)^kq^{\frac{3}{2}(k-\frac{1}{6})^2-\frac{1}{24}}}{(q)_\infty}\\
&\cong \oint \frac{dx_1}{2\pi i x_1}\frac{1}{x_1^{\frac{1}{2}}-x_1^{-\frac{1}{2}}}\sum_{\epsilon_2,\epsilon_3,\epsilon_7=\pm 1}\sum_{\ell_1\in \mathbb{Z}+\frac{1}{2}}\epsilon_2\epsilon_3\epsilon_7\; q^{-42(\ell_1-\frac{1}{2}(\frac{1}{2}\epsilon_2+\frac{1}{3}\epsilon_3+\frac{1}{7}\epsilon_7))^2+\frac{1}{168}}x_1^{\ell_1}\\
&\quad\times \frac{\sum_{\substack{k\in \mathbb{Z}\\ |k|> 6|\ell_1-\frac{1}{2}(\frac{1}{2}\epsilon_2 +\frac{1}{3}\epsilon_3+\frac{1}{7}\epsilon_7)|}}(-1)^kq^{\frac{3}{2}(k-\frac{1}{6})^2-\frac{1}{24}}}{(q)_\infty}\\
&= \oint \frac{dx_1}{2\pi i x_1}\frac{1}{x_1^{\frac{1}{2}}-x_1^{-\frac{1}{2}}}\sum_{\epsilon_2,\epsilon_3,\epsilon_7=\pm 1}\sum_{\ell_1\in \mathbb{Z}+\frac{1}{2}}\epsilon_2\epsilon_3\epsilon_7\; q^{-42(\ell_1-\frac{1}{2}(\frac{1}{2}\epsilon_2+\frac{1}{3}\epsilon_3+\frac{1}{7}\epsilon_7))^2+\frac{1}{168}}x_1^{\ell_1}\\
&\quad\times \frac{\sum_{\substack{k\in \mathbb{Z}\\ |k-\frac{1}{6}|> 6|\ell_1-\frac{1}{2}(\frac{1}{2}\epsilon_2 +\frac{1}{3}\epsilon_3+\frac{1}{7}\epsilon_7)|}}(-1)^kq^{\frac{3}{2}(k-\frac{1}{6})^2-\frac{1}{24}}}{(q)_\infty}.
\end{align*}
Comparing this last expression with \eqref{eq:naiveMinusSigma237}, we see that a regularization factor has been added to the naive non-convergent expression. What this regularization factor is doing is clear; it first multiplies $(q)_\infty$ to both the numerator and the denominator, and then restrict the range of summation from the rank 2 indefinite lattice $(\ell_1,k)\in (\mathbb{Z}+\frac{1}{2})\times \mathbb{Z}$ to a double cone $|k-\frac{1}{6}|> 6|\ell_1-\frac{1}{2}(\frac{1}{2}\epsilon_2 +\frac{1}{3}\epsilon_3+\frac{1}{7}\epsilon_7)|$ where the lattice is positive definite. 
\end{eg}
\begin{rmk}
While our expression of $\hat{Z}(-\Sigma(2,3,7))$ is very similar to that of \cite{CFS}, there is an important difference between the two. While we regularized the indefinite theta function by restricting it to a double cone, \cite{CFS} regularized it by restricting it to a one-sided cone $k-\frac{1}{6}>\frac{16}{3}|\ell_1-\frac{1}{2}(\frac{1}{2}\epsilon_2 +\frac{1}{3}\epsilon_3+\frac{1}{7}\epsilon_7)|$. The one-sided cone is slightly wider than the double cone, and they give the same answer. See Fig. \ref{fig:cones}.
\begin{figure}[ht]
    \centering
    \includegraphics[scale=0.7]{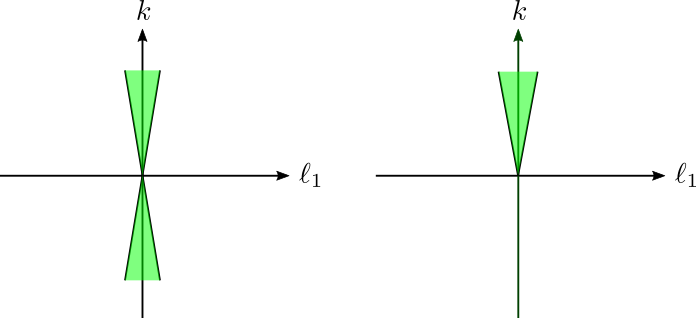}
    \caption{Double cone vs one-sided cone}
    \label{fig:cones}
\end{figure}
We believe it is more natural to use a double cone instead of a one-sided cone since it comes naturally from the regularized surgery formula. 
\end{rmk}

\begin{eg}[$-\Sigma(2,3,5)$]
The $3$-manifold $-\Sigma(2,3,5)$ has various descriptions, such as
\[
-\Sigma(2,3,5) = M(-1;\frac{1}{2},\frac{1}{3},\frac{1}{5}) = S^3_{+1}(\mathbf{3}_1^r).
\]
For us, the description as the $+1$-surgery on the right-handed trefoil is the most useful one. 
The naive plumbing formula gives
\begin{align*}
\hat{Z}(-\Sigma(2,3,5)) ``&\cong" \oint \frac{d x_1}{2\pi i x_1}\frac{d x_2}{2\pi i x_2}\frac{d x_3}{2\pi i x_3}\frac{d x_5}{2\pi i x_5} \frac{(x_2^{\frac{1}{2}}-x_2^{-\frac{1}{2}})(x_3^{\frac{1}{2}}-x_3^{-\frac{1}{2}})(x_5^{\frac{1}{2}}-x_5^{-\frac{1}{2}})}{x_1^{\frac{1}{2}}-x_1^{-\frac{1}{2}}}\\
&\quad\quad\times \sum_{\vec{\ell}\in M\mathbb{Z}^4 + \frac{1}{2}\vec{1}}q^{-(\vec{\ell}, M^{-1}\vec{\ell})}\vec{x}^{\vec{\ell}},
\end{align*}
where
\[
M = 
\begin{pmatrix}
-1 & 1 & 1 & 1 \\
1 & -2 & 0 & 0 \\
1 & 0 & -3 & 0 \\
1 & 0 & 0 & -5
\end{pmatrix}.
\]
Integrating out $x_2, x_3, x_5$, we get 
\begin{equation*}\label{eq:naiveMinusSigma235}
\hat{Z}(-\Sigma(2,3,5)) ``\cong"  \oint \frac{dx_1}{2\pi i x_1}\frac{1}{x_1^{\frac{1}{2}}-x_1^{-\frac{1}{2}}}\sum_{\epsilon_2,\epsilon_3,\epsilon_5=\pm 1}\sum_{\ell_1\in \mathbb{Z}+\frac{1}{2}}\epsilon_2\epsilon_3\epsilon_5\; q^{-30(\ell_1-\frac{1}{2}(\frac{1}{2}\epsilon_2+\frac{1}{3}\epsilon_3+\frac{1}{5}\epsilon_5))^2+\frac{1}{120}}x_1^{\ell_1}.
\end{equation*}
From the plumbing description of the right-handed trefoil knot, we have
\[
F_{\mathbf{3}_1^r}(x,q) \cong \oint \frac{d x_1}{2\pi i x_1}\frac{d x_2}{2\pi i x_2}\frac{d x_3}{2\pi i x_3}\frac{(x_2^{\frac{1}{2}}-x_2^{-\frac{1}{2}})(x_3^{\frac{1}{2}}-x_3^{-\frac{1}{2}})}{x_1^{\frac{1}{2}}-x_1^{-\frac{1}{2}}}\sum_{\vec{n}\in \mathbb{Z}^3\times\{0\}}q^{-(\vec{n}, M'\vec{n})-(\vec{n},\vec{1})}\vec{x}^{M'\vec{n}+\frac{1}{2}\vec{1}},
\]
where
\[
M' = 
\begin{pmatrix}
-1 & 1 & 1 & 1 \\
1 & -2 & 0 & 0 \\
1 & 0 & -3 & 0 \\
1 & 0 & 0 & -6
\end{pmatrix}.
\]
Using the regularized $+1$-surgery formula, we can deduce an expression for $\hat{Z}(-\Sigma(2,3,5))$ in terms of a indefinite theta function, just like the way we obtained such an expression for $\hat{Z}(-\Sigma(2,3,7))$. The result looks like
\begin{align*}
\hat{Z}(-\Sigma(2,3,5)) &\cong \oint \frac{dx_1}{2\pi i x_1}\frac{1}{x_1^{\frac{1}{2}}-x_1^{-\frac{1}{2}}}\sum_{\epsilon_2,\epsilon_3,\epsilon_5=\pm 1}\sum_{\ell_1\in \mathbb{Z}+\frac{1}{2}}\epsilon_2\epsilon_3\epsilon_5\; q^{-30(\ell_1-\frac{1}{2}(\frac{1}{2}\epsilon_2+\frac{1}{3}\epsilon_3+\frac{1}{5}\epsilon_5))^2+\frac{1}{120}}x_1^{\ell_1}\\
&\quad\times \frac{\sum_{\substack{k\in \mathbb{Z}\\ |k-\frac{1}{6}|> 6|\ell_1-\frac{1}{2}(\frac{1}{2}\epsilon_2 +\frac{1}{3}\epsilon_3+\frac{1}{5}\epsilon_5)|}}(-1)^kq^{\frac{3}{2}(k-\frac{1}{6})^2-\frac{1}{24}}}{(q)_\infty}.
\end{align*}
Again, we see that the indefinite theta function is regularized by restricting the range of summation to a double cone $|k-\frac{1}{6}|> 6|(\ell_1-\frac{1}{2}(\frac{1}{2}\epsilon_2 +\frac{1}{3}\epsilon_3+\frac{1}{5}\epsilon_5))|$. 
\end{eg}

\begin{eg}[$-\Sigma(2,3,13)$]
The $3$-manifold $-\Sigma(2,3,13)$ can be obtained as the $+\frac{1}{2}$-surgery on the left-handed trefoil knot. As a plumbed $3$-manifold, the plumbing graph has the linking matrix
\[
M = 
\begin{pmatrix}
1 & 1 & 1 & 1 & 0\\
1 & 2 & 0 & 0 & 0\\
1 & 0 & 3 & 0 & 0\\
1 & 0 & 0 & 6 & 1\\
0 & 0 & 0 & 1 & -2
\end{pmatrix}.
\]
Assuming that the regularized $+\frac{1}{2}$-surgery formula is valid, we can readily apply it to the expression of $F_{\mathbf{3}_1^l}$ coming from the plumbing description, and as a result we get the following
\begin{align*}
\hat{Z}(-\Sigma(2,3,13)) &\cong \oint\frac{dx_1}{2\pi i x_1} \frac{1}{x_1^{\frac{1}{2}}-x_1^{-\frac{1}{2}}}\sum_{\epsilon_2,\epsilon_3,\epsilon_{-2} = \pm 1}\sum_{\ell_1\in \mathbb{Z}+\frac{1}{2}}\epsilon_2\epsilon_3\epsilon_{-2}q^{-78(\ell_1 - \frac{1}{2}(\frac{1}{2}\epsilon_2+\frac{1}{3}\epsilon_3+\frac{1}{13}\epsilon_{-2}))^2-\frac{263}{312}}x_1^{\ell_1}\\
&\quad\times \frac{\sum_{\substack{k\in \mathbb{Z} \\ |k-\frac{1}{10}|>6|\ell_1 -\frac{1}{2}(\frac{1}{2}\epsilon_2+\frac{1}{3}\epsilon_3+\frac{1}{13}\epsilon_{-2})|}} (-1)^k q^{\frac{5}{2}(k-\frac{1}{10})^2-\frac{1}{40}}}{f(-q^2,-q^3)}\\
&= 1+q^3-q^4+q^5-q^7+2q^8-2q^9+2q^{10} +O(q^{11}).
\end{align*}
\end{eg}

\begin{eg}[$S^3_{-2}(\mathbf{4}_1)$]
The $-2$-surgery on the figure-eight knot is a Seifert manifold 
\[
S^3_{-2}(\mathbf{4}_1) = M(1;-\frac{1}{2},-\frac{1}{4},-\frac{1}{5}).
\]
As a plumbed $3$-manifold, the plumbing graph has the linking matrix
\[
M = 
\begin{pmatrix}
1 & 1 & 1 & 1\\
1 & 2 & 0 & 0\\
1 & 0 & 4 & 0\\
1 & 0 & 0 & 5
\end{pmatrix}.
\]
We have experimentally found the double cone, and the resulting expression for $\hat{Z}(S_{-2}^3(\mathbf{4}_1))$ is given by
\begin{align*}
\hat{Z}(S^3_{-2}(\mathbf{4}_1)) &\cong \oint \frac{dx_1}{2\pi i x_1}\frac{1}{x_1^{\frac{1}{2}}-x_1^{-\frac{1}{2}}}\sum_{\epsilon_2,\epsilon_4,\epsilon_5=\pm 1}\sum_{\ell_1\in \mathbb{Z}+\frac{1}{2}}\epsilon_2\epsilon_4\epsilon_5\; q^{-\frac{40}{2}(\ell_1-\frac{1}{2}(\frac{1}{2}\epsilon_2+\frac{1}{4}\epsilon_4+\frac{1}{5}\epsilon_5))^2+\frac{1}{80}}x_1^{\ell_1}\\
&\quad\times \frac{\sum_{\substack{k\in \mathbb{Z}\\ |k-\frac{1}{6}|> 4|\ell_1-\frac{1}{2}(\frac{1}{2}\epsilon_2 +\frac{1}{4}\epsilon_4+\frac{1}{5}\epsilon_5)|}}(-1)^kq^{\frac{3}{2}(k-\frac{1}{6})^2-\frac{1}{24}}}{(q)_\infty}.
\end{align*}
To separate the contribution of each $b$, we just need to add the characteristic function for the following condition: 
\[
\begin{pmatrix}\ell_1 \\ \epsilon_2 \\ \epsilon_4 \\ \epsilon_5\end{pmatrix} \in M \mathbb{Z}^4 + b.
\]
\end{eg}

\begin{eg}[$S^3_{-3}(\mathbf{4}_1)$]
The $-3$-surgery on the figure-eight knot is a Seifert manifold
\[
S^3_{-3}(\mathbf{4}_1) = M(1;-\frac{1}{3},-\frac{1}{3},-\frac{1}{4})
\]
As a plumbed $3$-manifold, the plumbing graph has the linking matrix
\[
M = 
\begin{pmatrix}
1 & 1 & 1 & 1\\
1 & 3 & 0 & 0\\
1 & 0 & 3 & 0\\
1 & 0 & 0 & 4
\end{pmatrix}.
\]
We have experimentally found the double cone, and the resulting expression for $\hat{Z}(S_{-3}^3(\mathbf{4}_1))$ is given by
\begin{align*}
\hat{Z}(S^3_{-3}(\mathbf{4}_1)) &\cong \oint \frac{dx_1}{2\pi i x_1}\frac{1}{x_1^{\frac{1}{2}}-x_1^{-\frac{1}{2}}}\sum_{\epsilon_{3a},\epsilon_{3b},\epsilon_4=\pm 1}\sum_{\ell_1\in \mathbb{Z}+\frac{1}{2}}\epsilon_{3a}\epsilon_{3b}\epsilon_4\; q^{-\frac{36}{3}(\ell_1-\frac{1}{2}(\frac{1}{3}\epsilon_{3a}+\frac{1}{3}\epsilon_{3b}+\frac{1}{4}\epsilon_4))^2+\frac{1}{48}}x_1^{\ell_1}\\
&\quad\times \frac{\sum_{\substack{k\in \mathbb{Z}\\ |k-\frac{1}{6}|> 3|\ell_1-\frac{1}{2}(\frac{1}{3}\epsilon_{3a} +\frac{1}{3}\epsilon_{3b}+\frac{1}{4}\epsilon_4)|}}(-1)^kq^{\frac{3}{2}(k-\frac{1}{6})^2-\frac{1}{24}}}{(q)_\infty}.
\end{align*}
To separate the contribution of each $b$, we just need to add the characteristic function for the following condition: 
\[
\begin{pmatrix}\ell_1 \\ \epsilon_{3a} \\ \epsilon_{3b} \\ \epsilon_4\end{pmatrix} \in M \mathbb{Z}^4 + b.
\]
\end{eg}

\begin{eg}[$-\Sigma(2,3,7)$ with another degree $3$ node]
As our last example, let's see what happens to the regularization factor if we create a new degree $3$ node to the plumbing graph of $-\Sigma(2,3,7)$ by applying Neumann moves twice. (As a reminder, Neumann moves are a set of moves on the plumbing graph that preserves the resulting 3-manifold. See Section 4 of \cite{GM} for a review.)
The linking matrix of the resulting plumbing graph is
\[
M = 
\begin{pmatrix}
1 & 1 & 1 & 1 & 0 & 0\\
1 & 2 & 0 & 0 & 0 & 0\\
1 & 0 & 3 & 0 & 0 & 0\\
1 & 0 & 0 & 5 & 1 & 1\\
0 & 0 & 0 & 1 & -1 & 0\\
0 & 0 & 0 & 1 & 0 & -1
\end{pmatrix}.
\]
As a result, after multiplying $(q)_\infty$ in both the numerator and the denominator, the indefinite theta function is of signature $+,-,-$. With a bit of experiment, we found an appropriate double cone, and the resulting expression is given below
\begin{align*}
\hat{Z}(-\Sigma(2,3,7)) &\cong \oint \frac{dx_1}{2\pi i x_1}\frac{dx_5}{2\pi i x_5}\frac{1}{x_1^{\frac{1}{2}}-x_1^{-\frac{1}{2}}}\frac{1}{x_5^{\frac{1}{2}}-x_5^{-\frac{1}{2}}}\sum_{\epsilon_2,\epsilon_3,\epsilon_{-1a},\epsilon_{-1b} =\pm 1}\sum_{\ell_1, \ell_5\in \mathbb{Z}+\frac{1}{2}}\epsilon_2\epsilon_3\epsilon_{-1a}\epsilon_{-1b}\\
&\quad\times q^{(\ell_1 - \frac{1}{2}(\frac{1}{2}\epsilon_2+\frac{1}{3}\epsilon_3), \ell_5 -\frac{1}{2}(-\epsilon_{-1a}-\epsilon_{-1b}))\cdot \begin{psmallmatrix}-42 & 6 \\ 6 & -1\end{psmallmatrix} \cdot (\ell_1 - \frac{1}{2}(\frac{1}{2}\epsilon_2+\frac{1}{3}\epsilon_3), \ell_5 -\frac{1}{2}(-\epsilon_{-1a}-\epsilon_{-1b}))^t +\frac{1}{24}}\\
&\quad\times x_1^{\ell_1}x_5^{\ell_5}\frac{\sum_{\substack{k\in \mathbb{Z} \\ |k-\frac{1}{6}| > 6|\ell_1-\frac{1}{2}(\frac{1}{2}\epsilon_2+\frac{1}{3}\epsilon_3)| \\ |k-\frac{1}{6}| > |\ell_5 -\frac{1}{2}(-\epsilon_{-1a}-\epsilon_{-1b})|}}(-1)^k q^{\frac{3}{2}(k-\frac{1}{6})-\frac{1}{24}}}{(q)_\infty}.
\end{align*}
That is, the double cone is determined by the two inequalities $|k-\frac{1}{6}| > 6|\ell_1-\frac{1}{2}(\frac{1}{2}\epsilon_2+\frac{1}{3}\epsilon_3)|$ and $|k-\frac{1}{6}| > |\ell_5 -\frac{1}{2}(-\epsilon_{-1a}-\epsilon_{-1b})|$. 
Compare this with Example \ref{eg:minusSigma237}. It is a very interesting problem to study how the double cone behaves under Neumann moves in general. 
\end{eg}

\section{Future directions}
We conclude this paper by collecting some interesting questions for further study. 
\begin{itemize}
    \item Is there a way to derive the appearance of both highest weight and lowest weight Verma modules from complex Chern-Simons theory?
    \item What is the most general class of links for which the inverted state sum method works? Does it contain all fibered knots?
    \item What's the $\infty$-surgery formula for $\hat{Z}$? (Question \ref{qn:infinitySurgery})
    \item Is there a way to obtain $F_{m(\mathbf{5}_2)}(x,q)$ directly from a state sum model? More generally, is there a natural way to regularize the $q$-series in the context of state sum?
    \item Can we invert the cyclotomic expansions for $\mathfrak{gl}_N$ knot invariants studied by \cite{BG}?
    \item When written as an inverted Habiro series, $F_K$ naturally has poles at $x=q^n$. What's the meaning of their residues?
    \item Given a general plumbing graph, what's the regularization factor gives rise to $\hat{Z}$? In particular, how does the double cone behave under Neumann moves?
\end{itemize}

\appendix
\section{All fibered knots up to 10 crossings}\label{sec:fiberedKnotsInversionData}
According to KnotInfo \cite{KnotInfo}, there are 117 fibered knots up to 10 crossings. Using their minimum braid representatives in Knot Atlas \cite{KnotAtlas}, we see that 74 of them are homogeneous braids. Since they are covered by Theorem \ref{thm:mainTheorem}, we can focus on the other 43 of them which are non-homogeneous braids. It turns out, all of them can be covered by Theorem \ref{thm:beyondHomogeneousBraids}. We summarize the inversion data (assignment of $+$'s and $-$'s to each segment so that the state sum is absolutely convergent) in Table \ref{tab:fiberedKnotsInversionData}. 
\begin{table}[ht]
    \centering
    \begin{tabular}{|c|c|c|}
        \hline
        Knot & Braid & Inversion data \\
        \hline\hline
        $\mathbf{8}_{20}$ & $\sigma_1^{-1}\sigma_2^{-3}\sigma_1^{-1}\sigma_2^3$ & $\X\NE\E\NW\X\W\W\W$\\
        $\mathbf{8}_{21}$ & $\sigma_1^{-2}\sigma_2^2\sigma_1^{-1}\sigma_2^{-3}$ & $\X\X\W\W\X\NE\E\NW$\\
        $\mathbf{9}_{42}$ & $\sigma_1\sigma_2^{-1}\sigma_1\sigma_3^{-2}\sigma_2^{-1}\sigma_3^{3}$ & $\n\E\n\NE\NW\E\W\W\W$\\
        $\mathbf{9}_{44}$ & $\sigma_1\sigma_2^{-1}\sigma_1\sigma_3^2\sigma_2^{-1}\sigma_3^{-3}$ & $\n\E\n\W\W\E\NE\E\NW$\\
        $\mathbf{9}_{45}$ & $\sigma_1^{-1}\sigma_2\sigma_1^{-1}\sigma_3^{-1}\sigma_2^{-1}\sigma_3\sigma_2^{-1}\sigma_3^{-2}$ & $\X\W\X\NW\X\W\X\NE\E$\\
        $\mathbf{9}_{48}$ & $\sigma_1^{-1}\sigma_2\sigma_3^{-1}\sigma_2\sigma_1^{-1}\sigma_3\sigma_2\sigma_3^{-1}\sigma_2\sigma_3^2$ & $\X\W\n\W\X\n\W\n\W\n\n$\\
        $\mathbf{10}_{60}$ & $\sigma_1^{-1}\sigma_2\sigma_1^{-1}\sigma_2^2\sigma_3^{-1}\sigma_2\sigma_3^{-1}\sigma_2^{-1}\sigma_4^{-1}\sigma_3\sigma_4^{-1}$ & $\X\W\X\W\W\E\W\NW\X\E\NE\X$\\
        $\mathbf{10}_{69}$ & $\sigma_1\sigma_2^{-1}\sigma_3\sigma_2^{-1}\sigma_4\sigma_1\sigma_3\sigma_2^{-1}\sigma_4^{-1}\sigma_3\sigma_4^2$ & $\n\E\W\E\n\n\W\E\n\W\n\n$\\
        $\mathbf{10}_{73}$ & $\sigma_1^{-1}\sigma_2\sigma_1^{-1}\sigma_2\sigma_3^{-1}\sigma_2\sigma_4^{-1}\sigma_3^{-1}\sigma_4\sigma_3^{-1}\sigma_4^{-2}$ & $\X\W\X\W\E\W\X\E\X\E\X\X$\\
        $\mathbf{10}_{75}$ & $\sigma_1\sigma_2^{-1}\sigma_1\sigma_2^{-1}\sigma_3\sigma_2^{-2}\sigma_4\sigma_3^{-1}\sigma_2\sigma_4\sigma_3$ & $\n\E\n\E\W\E\E\n\NE\n\W\NW$\\
        $\mathbf{10}_{78}$ & $\sigma_1^{-2}\sigma_2\sigma_1^{-1}\sigma_3^{-1}\sigma_2\sigma_4^{-1}\sigma_3^{-1}\sigma_4\sigma_3^{-1}\sigma_4^{-2}$ & $\X\X\W\X\E\W\X\E\X\E\X\X$\\
        $\mathbf{10}_{81}$ & $\sigma_1^2\sigma_2^{-1}\sigma_1\sigma_3\sigma_2^2\sigma_4^{-1}\sigma_3^{-3}\sigma_4^{-1}$ & $\n\n\E\n\NE\n\n\X\E\E\NW\E$\\
        $\mathbf{10}_{89}$ & $\sigma_1^{-1}\sigma_2^{-1}\sigma_3\sigma_2^{-1}\sigma_4^{-1}\sigma_1^{-1}\sigma_3^{-1}\sigma_2\sigma_3\sigma_4^{-1}\sigma_3\sigma_4^{-1}$ & $\X\X\W\X\E\X\NE\W\NW\E\W\E$\\
        $\mathbf{10}_{96}$ & $\sigma_1\sigma_2^{-1}\sigma_3\sigma_2^{-1}\sigma_4\sigma_1\sigma_2^{-1}\sigma_3\sigma_2^{-1}\sigma_4\sigma_3\sigma_4^{-1}$ & $\n\E\W\E\n\n\E\W\E\n\W\n$\\
        $\mathbf{10}_{105}$ & $\sigma_1^2\sigma_2^{-1}\sigma_1\sigma_3\sigma_2^2\sigma_4^{-1}\sigma_3^{-1}\sigma_2\sigma_3^{-1}\sigma_4^{-1}$ & $\n\n\E\n\NE\n\n\X\E\n\NW\E$\\
        $\mathbf{10}_{107}$ & $\sigma_1^{-2}\sigma_2\sigma_1^{-1}\sigma_3\sigma_2^2\sigma_4^{-1}\sigma_3\sigma_2^{-1}\sigma_3\sigma_4^{-1}$ & $\X\X\W\X\n\W\NE\E\W\NW\n\E$\\
        $\mathbf{10}_{110}$ & $\sigma_1^{-1}\sigma_2\sigma_1^{-1}\sigma_3^{-1}\sigma_2^{-3}\sigma_4\sigma_3\sigma_2^{-1}\sigma_3\sigma_4$ & $\X\W\X\NW\X\X\X\n\W\X\NE\W$\\
        $\mathbf{10}_{115}$ & $\sigma_1\sigma_2^{-1}\sigma_1\sigma_3\sigma_2^2\sigma_4^{-1}\sigma_3^{-1}\sigma_2\sigma_3^{-2}\sigma_4^{-1}$ & $\n\E\n\NE\n\n\X\E\n\E\NW\E$\\
        $\mathbf{10}_{125}$ & $\sigma_1^{-1}\sigma_2^{-3}\sigma_1^{-1}\sigma_2^5$ & $\X\NE\E\NW\X\W\W\W\W\W$\\
        $\mathbf{10}_{126}$ & $\sigma_1^{-1}\sigma_2^3\sigma_1^{-1}\sigma_2^{-5}$ & $\X\W\W\W\X\NE\E\E\E\NW$\\
        $\mathbf{10}_{127}$ & $\sigma_1^{-2}\sigma_2^2\sigma_1^{-1}\sigma_2^{-5}$ & $\X\X\W\W\X\NE\E\E\E\NW$\\
        $\mathbf{10}_{132}$ & $\sigma_1^{-1}\sigma_2^{-1}\sigma_3^{-1}\sigma_2\sigma_1^3\sigma_3^{-2}\sigma_2^{-1}\sigma_1^{-1}$ & $\X\NE\NE\n\W\W\W\E\NW\NW\X$\\
        $\mathbf{10}_{133}$ & $\sigma_1^{-2}\sigma_2\sigma_1^{-1}\sigma_2^{-1}\sigma_3^2\sigma_2^{-1}\sigma_3^{-3}$ & $\X\X\X\X\X\W\W\X\NE\E\NW$\\
        $\mathbf{10}_{136}$ & $\sigma_1\sigma_2^{-1}\sigma_1\sigma_3^2\sigma_2^{-1}\sigma_3^{-1}\sigma_4\sigma_3^{-1}\sigma_4$ & $\n\E\n\W\W\E\NE\W\NW\n$\\
        $\mathbf{10}_{137}$ & $\sigma_1\sigma_2^{-1}\sigma_1\sigma_3^{-2}\sigma_2^{-1}\sigma_3\sigma_4^{-1}\sigma_3\sigma_4^{-1}$ & $\n\E\n\NE\NW\E\W\E\W\E$\\
        $\mathbf{10}_{140}$ & $\sigma_1^{-1}\sigma_2^{-1}\sigma_3^{-1}\sigma_2\sigma_1^3\sigma_3^{-1}\sigma_2^{-1}\sigma_1^{-2}$ & $\X\NE\NE\n\W\W\W\NW\NW\X\X$\\
        $\mathbf{10}_{141}$ & $\sigma_1^{-2}\sigma_2^{-3}\sigma_1^{-1}\sigma_2^4$ & $\X\X\NE\E\NW\X\W\W\W\W$\\
        $\mathbf{10}_{143}$ & $\sigma_1^{-2}\sigma_2^3\sigma_1^{-1}\sigma_2^{-4}$ & $\X\X\W\W\W\X\NE\E\E\NW$\\
        $\mathbf{10}_{145}$ & $\sigma_1^{-1}\sigma_2^{-1}\sigma_3\sigma_2^{-1}\sigma_1^{-1}\sigma_3^{-1}\sigma_2^{-1}\sigma_3\sigma_2^{-1}\sigma_3^{-2}$ & $\X\X\X\X\X\X\X\X\X\X\X$\\
        $\mathbf{10}_{148}$ & $\sigma_1^{-1}\sigma_2\sigma_1^{-1}\sigma_2^2\sigma_1^{-1}\sigma_2^{-4}$ & $\X\W\X\W\W\X\NE\E\E\NW$\\
        $\mathbf{10}_{149}$ & $\sigma_1^{-2}\sigma_2\sigma_1^{-1}\sigma_2\sigma_1^{-1}\sigma_2^{-4}$ & $\X\X\W\X\W\X\NE\E\E\NW$\\
        $\mathbf{10}_{150}$ & $\sigma_2\sigma_1\sigma_3^{-1}\sigma_2^{-1}\sigma_1\sigma_3^2\sigma_2^{-1}\sigma_3^3$ & $\n\n\NW\E\n\W\W\E\W\W\NE$\\
        $\mathbf{10}_{151}$ & $\sigma_2^{-1}\sigma_1\sigma_3\sigma_2^{-1}\sigma_1\sigma_3^{-2}\sigma_2\sigma_3^3$ & $\E\n\W\E\n\NE\E\n\NW\W\W$\\
        $\mathbf{10}_{153}$ & $\sigma_1^{-3}\sigma_2^{-1}\sigma_1^{-2}\sigma_3\sigma_2^3\sigma_3$ & $\X\X\X\X\X\X\NE\W\W\W\NW$\\
        $\mathbf{10}_{154}$ & $\sigma_1\sigma_2^3\sigma_1\sigma_3\sigma_2\sigma_3^{-1}\sigma_2\sigma_3^2$ & $\n\n\n\n\n\n\n\n\n\n\n$\\
        $\mathbf{10}_{155}$ & $\sigma_1\sigma_2^{-2}\sigma_1\sigma_2^{-2}\sigma_1\sigma_2^3$ & $\n\E\E\n\E\E\n\NW\W\NE$\\
        $\mathbf{10}_{156}$ & $\sigma_1^{-1}\sigma_2^{-1}\sigma_3\sigma_2^{-1}\sigma_1^{-1}\sigma_3^2\sigma_2\sigma_3^{-3}$ & $\X\X\W\X\X\W\NE\W\E\E\NW$\\
        $\mathbf{10}_{157}$ & $\sigma_1^2\sigma_2^{-1}\sigma_1\sigma_2^{-1}\sigma_1^2\sigma_2^3$ & $\n\n\E\n\E\n\n\NW\W\NE$\\
        $\mathbf{10}_{158}$ & $\sigma_1\sigma_2\sigma_3^{-1}\sigma_2\sigma_1\sigma_3^2\sigma_2^{-1}\sigma_3^{-3}$ & $\n\n\E\n\n\NW\W\E\NE\E\E$\\
        $\mathbf{10}_{159}$ & $\sigma_1^{-2}\sigma_2^2\sigma_1^{-1}\sigma_2\sigma_1^{-1}\sigma_2^{-3}$ & $\X\X\W\W\X\W\X\NE\E\NW$\\
        $\mathbf{10}_{160}$ & $\sigma_1^{-1}\sigma_2\sigma_3^{-1}\sigma_2\sigma_1^{-1}\sigma_3^2\sigma_2\sigma_3^3$ & $\X\W\n\W\X\n\n\W\n\n\n$\\
        $\mathbf{10}_{161}$ & $\sigma_1^{-2}\sigma_2^{-2}\sigma_1^{-1}\sigma_2\sigma_1^{-1}\sigma_2^{-3}$ & $\X\X\X\X\X\X\X\X\X\X$\\
        $\mathbf{10}_{163}$ & $\sigma_1\sigma_2^2\sigma_3^{-1}\sigma_2\sigma_1\sigma_3^{-2}\sigma_2^{-1}\sigma_3^2$ & $\n\n\n\E\n\n\E\NW\E\W\NE$\\
        \hline
    \end{tabular}
    \caption{Fibered knots up to 10 crossings that are possibly not homogeneous braid knots}
    \label{tab:fiberedKnotsInversionData}
\end{table}
In the table, each inversion datum shows how each elementary braid in the braid word looks like. The green segments represent the ones labelled with $-$ signs and the black segments represent the ones labelled with $+$ signs. 
For example, the inversion data for $\mathbf{8}_{20}$ can be translated into the diagram in Fig. \ref{fig:8_20}. 
\begin{figure}[ht]
    \centering
    \includegraphics[scale=0.5]{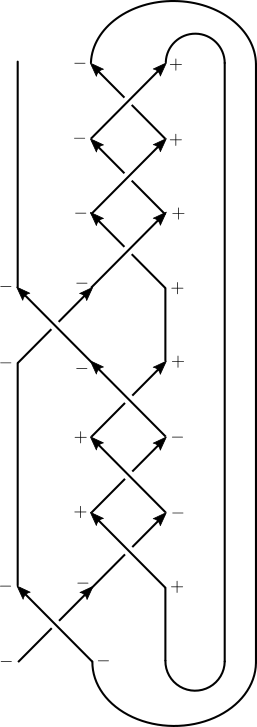}
    \caption{The $\mathbf{8}_{20}$ knot}
    \label{fig:8_20}
\end{figure}
As explained in Theorem \ref{thm:beyondHomogeneousBraids}, the overall sign is given by $(-1)^s$, where $s$ is the number of closed cycles of the set of $-$-signed segments as a simple multi-cycle. Whenever the $-$-signed strands cross each other, one should basically ignore the crossing, as the multi-cycle follows the same strand on such crossings. In case of the diagram in Fig. \ref{fig:8_20}, the $-$-signed simple multi-cycle consists of one open path and one closed cycle, so $(-1)^s = -1$ in this case.

\bibliography{ref}
\bibliographystyle{alpha}

\end{document}